\input ./style/arxiv-general.cfg
\documentclass[MSNbibl,number,citesort,dvips]{arxbj}
\makeatletter
   \@ifpackageloaded{graphicx}{}{\usepackage{graphicx}}
\makeatother

\volume{21}
\issue{4}
\pubyear{2015}
\firstpage{2484}
\lastpage{2512}
\doi{10.3150/14-BEJ652} 
\docsubty{FLA}

\makeatletter
\newcommand{\rrvert}{\vert}
\newcommand{\llvert}{\vert}
\def\1{\mathbf{1}}
\def\I{\mathbf{I}}
\def\R{\mathbb{R}}
\def\I{\mathbf{I}}
\def\E{\mathbb{E}}
\def\P{\mathbb{P}}

\newtheorem{lem}{Lemma}
\newtheorem{prop}[lem]{Proposition}
\newtheorem{thmm}[lem]{Theorem}
\newtheorem{cor}[lem]{Corollary}
\newproclaim{defn}[lem]{Definition}
\newproclaim{notat}[lem]{Notation}
\newremark{rem}[lem]{Remark}
\newremark{prob}[lem]{Problem}
\newremark{ex}[lem]{Example}
\makeatother

\begin{document}
\begin{frontmatter}

\title{Size-biased permutation of a finite sequence with independent and identically distributed terms}
\runtitle{S.b.p. of a finite i.i.d. sequence}

\begin{aug}
\author[A]{\inits{J.}\fnms{Jim}~\snm{Pitman}\thanksref{A}\ead[label=e1]{pitman@stat.berkeley.edu}}
\and
\author[B]{\inits{N.M.}\fnms{Ngoc M.}~\snm{Tran}\corref{}\thanksref{B}\ead[label=e2]{ntran@math.utexas.edu}}
\address[A]{Department of Statistics, UC Berkeley, CA 94720, USA.
\printead{e1}}
\address[B]{Department of Mathematics, UT Austin, TX 78712, USA.
\printead{e2}}
\end{aug}

\received{\smonth{10} \syear{2012}}
\revised{\smonth{3} \syear{2014}}

%
\begin{abstract}
This paper focuses on the size-biased permutation of $n$ independent
and identically distributed (i.i.d.) positive random variables. This is
a finite dimensional analogue of the size-biased permutation of ranked
jumps of a subordinator studied in Perman--Pitman--Yor (PPY)
[\textit{Probab. Theory Related Fields} \textbf{92} (1992) \mbox{21--39}],
as well as a special form of \emph{induced order
statistics}
[\textit{Bull. Inst. Internat. Statist.} \textbf{45} (1973) 295--300;
\textit{Ann. Statist.} \textbf{2} (1974) 1034--1039]. This intersection grants us
different tools for deriving distributional properties. Their
comparisons lead to new results, as well as simpler proofs of existing
ones. Our main contribution, Theorem~\ref{thm} in Section~\ref{sec:ppp}, describes the asymptotic distribution of the last few terms
in a finite i.i.d. size-biased permutation via a Poisson coupling with
its few smallest order statistics.
\end{abstract}

%
\begin{keyword}
\kwd{induced order statistics}
\kwd{Kingman paint box}
\kwd{Poisson--Dirichlet}
\kwd{size-biased permutation}
\kwd{subordinator}
\end{keyword}
%
\end{frontmatter}

\section{Introduction}\label{sec1}
Let $x = (x(1), x(2), \ldots)$ be a positive sequence with finite sum
$t = \sum_{i=1}^\infty x(i)$. Its size-biased permutation (s.b.p.) is
the same sequence presented in a random order $(x({\sigma_1}),
x({\sigma
_2}), \ldots)$, where $\mathbb{P}(\sigma_1 = i) = \frac{x(i)}{t}$,
and for $k$ distinct indices $i_1, \ldots, i_k$,
%
\begin{equation}
\mathbb{P}(\sigma_k = i_k| \sigma_1 =
i_1, \ldots, \sigma_{k-1} = i_{k-1}) =
\frac{x({i_k})}{t - (x({i_1}) + \cdots+
x({i_{k-1}}))}. \label{eqn:sbp}
\end{equation}
An index $i$ with bigger `size' $x(i)$ tends to appear earlier in the
permutation, hence the name size-biased. 
Size-biased permutation of a random sequence is defined by conditioning
on the sequence values.

One of the earliest occurrences of size-biased permutation is in social
choice theory. For fixed sequence length $n$, the goal is to infer the
$x(i)$ given multiple observations from the random permutation defined
by (\ref{eqn:sbp}). Here the $x(i)$ are the relative scores or
desirabilities of the candidates, and (\ref{eqn:sbp}) models the
distribution of their rankings in an election. Now known as the
Plackett--Luce model, it has wide applications \cite{plackett,plmeier,volkmer}. 

Around the same time, biologists in population genetics were interested
in inferring the distribution of alleles in a population through
sampling. In these applications, $x(i)$ is the abundance and $x(i)/t$
is the relative abundance of the $i$th species \cite{en78}. Size-biased
permutation models the outcome of successive sampling, where one
samples without replacement from the population and records the
abundance of newly discovered species in the order that they appear. To
account for the occurrence of new types of alleles through mutation and
migration, they considered random abundance sequences and did not
assume an upper limit to the number of possible types. Patil and
Taillie \cite{pt77} coined the term \emph{size-biased random
permutation} to describe i.i.d. sampling from a random discrete
distribution. The earliest work along this vein is perhaps that of
McCloskey \cite{mc65}, who obtained results on the size-biased
permutation of ranked jumps in a certain Poisson point process
(p.p.p.). The distribution of this ranked sequence is now known as the
Poisson--Dirichlet distribution $\operatorname{PD}(0, \theta)$. The distribution of its
size-biased permutation is the $\operatorname{GEM}(0, \theta)$ distribution. See
Section~\ref{subsec:parallel} for their definitions. This work was
later generalized by Perman, Pitman and Yor \cite{pitmanyor}, who
studied size-biased permutation of ranked jumps of a subordinator; see
Section~\ref{sec:sbp}.

Size-biased permutation of finite sequences appear naturally through
the study of partition structure. Kingman \cite{ki78a,ki78b} initiated
this theory to explain the Ewens sampling formula, which gives the
joint distribution of $n$ independent size-biased picks from a
$\operatorname{PD}(0,\theta)$ distribution. The theory of partition structure, in
particular that of exchangeable partitions, is closely related to the
$\operatorname{GEM}$ and Poisson--Dirichlet distributions. We briefly mention related
results in Section~\ref{subsec:invariance}. For details, see \cite{csp}, Sections 2, 3.

This paper focuses on \emph{finite i.i.d. size-biased permutation},
that is, the size-biased permutation of $n$ independent and identically
distributed (i.i.d.) random variables from some distribution $F$ on
$(0, \infty)$. Our setting is a finite dimensional analogue of the
size-biased permutation of ranked jumps of a subordinator studied in
\cite{pitmanyor}, as well as a special form of \emph{induced order
statistics} \cite{david73,bhat74}. This intersection grants us
different tools for deriving distributional properties. Their
comparison lead to new results, as well as simpler proofs of existing
ones. By considering size-biased permutation of i.i.d. triangular
arrays, we derive convergence in distribution of the remaining $u$
fraction in a successive sampling scheme. This provides alternative
proofs to similar statements in the successive sampling literature. Our
main contribution, Theorem~\ref{thm} in Section~\ref{sec:ppp},
describes the asymptotic distribution of the last few terms in a finite
i.i.d. size-biased permutation via a Poisson coupling with its few
smallest order statistics.

\subsection{Organization}
We derive joint and marginal distribution of finite i.i.d. size-biased
permutation in Section~\ref{subsec:joint.ppy} through a Markov chain,
and re-derive them in Section~\ref{sec:ios} using induced order
statistics. Section~\ref{sec:sbp} connects our setting and its infinite
version of \cite{pitmanyor}. As the sequence length tends to infinity,
we derive asymptotics of the last $u$ fraction of finite i.i.d.
size-biased permutation in Section~\ref{sec:asymp}, and that of the
first few terms in Section~\ref{sec:ppp}.

\subsection{Notation} We shall write $\operatorname{gamma}(a,\lambda)$ for a Gamma
distribution whose density at $x$ is $\lambda^a x^{a-1}\mathrm{e}^{-\lambda
x}/\break\Gamma(a)$ for $x > 0$, and $\operatorname{beta}(a,b)$ for the Beta distribution
whose density at $x$ is $\frac{\Gamma(a+b)}{\Gamma(a)\Gamma
(b)}x^{a-1}(1-x)^{b-1}$ for $x \in(0,1)$. For an ordered sequence, not
necessarily order statistics, $(Y_n(k), k = 1, \ldots, n)$, let
$Y^{\mathrm{rev}}_n(k) = Y_n(n-k+1)$ be the same sequence presented in reverse.
For order statistics, we write $Y^\uparrow$ for the increasing
sequence, and $Y^\downarrow$ for the decreasing sequence. Throughout
this paper, unless otherwise indicated, we use $X_n = (X_n(1), \ldots,
X_n(n))$ for the underlying i.i.d. sequence with $n$ terms, and
$(X_n[1], \ldots, X_n[n])$ for its size-biased permutation. To avoid
having to list out the terms, it is also sometimes convenient to write
$X^\ast_n = (X_n^\ast(1), \ldots, X_n^\ast(n))$ for the size-biased
permutation of $X_n$.

\section{Markov property and stick-breaking}\label{subsec:joint.ppy}
Assume that $F$ has density $\nu_1$. Let $T_{n-k} = X_n[k+1] + \cdots+
X_n[n]$ denote the sum of the last $n-k$ terms in an i.i.d. size-biased
permutation of length $n$. We first derive joint distribution of the
first $k$ terms
$X_n[1], \ldots, X_n[k]$.

\begin{prop}[(Barouch--Kaufman \cite{bkauf})] For $1 \leq k \leq n$, let
$\nu_k$ be the density of $S_k$, the sum of $k$ i.i.d. random variables
with distribution $F$. Then
%
\begin{eqnarray}
&& \P\bigl(X_n[1] \in \mathrm{d}x_1, \ldots, X_n[k]
\in \mathrm{d}x_k\bigr)
\nonumber
\\
&&\quad=  \frac{n!}{(n-k)!} \Biggl(\prod_{j=1}^kx_j
\nu_1(x_j) \,\mathrm{d}x_j \Biggr) \int
_0^\infty\nu_{n-k}(s) \prod
_{j=1}^k(x_j + \cdots+
x_k+s)^{-1} \,\mathrm{d}s \label{eqn:joint.xnk}
\\
&&\quad=  \frac{n!}{(n-k)!} \Biggl(\prod_{j=1}^kx_j
\nu_1(x_j) \,\mathrm{d}x_j \Biggr) \E \Biggl(\prod
_{j=1}^k\frac{1}{x_j+\cdots
+x_k+S_{n-k}} \Biggr)
\label{eqn:sbp.joint.e}.
\end{eqnarray}
\end{prop}

\begin{pf}
Let $\sigma$ denote the random permutation on $n$ letters defined by
size-biased permutation as in (\ref{eqn:sbp}). Then there are $\frac
{n!}{(n-k)!}$ distinct possible values for $(\sigma_1, \ldots, \sigma
_k)$. By exchangeability of the underlying i.i.d. random variables
$X_n(1), \ldots, X_n(n)$, it is sufficient to consider $\sigma_1 = 1,
\ldots, \sigma_k = k$. Note that
\[
\P \Biggl(\bigl(X_n(1), \ldots, X_n(k)\bigr) \in
\mathrm{d}x_1 \cdots \mathrm{d}x_k, \sum_{j=k+1}^nX_n(j)
\in \mathrm{d}s \Biggr) = \nu_{n-k}(s) \,\mathrm{d}s \prod_{j=1}^k
\nu _1(x_j) \,\mathrm{d}x_j.
\]
Thus, restricted to $\sigma_1 = 1, \ldots, \sigma_k = k$, the
probability of observing $(X_n[1], \ldots, X_n[k]) \in \mathrm{d}x_1 \cdots
\mathrm{d}x_k$ and $T_{n-k} \in \mathrm{d}s$ is precisely
\[
\frac{x_1}{x_1 + \cdots+ x_k + s} \frac{x_2}{x_2 + \cdots+ x_k + s} \cdots\frac{x_k}{x_k + s}
\nu_{n-k}(s) \Biggl(\prod_{j=1}^k
\nu _1(x_j) \,\mathrm{d}x_j \Biggr) \,\mathrm{d}s.
\]
By summing over $\frac{n!}{(n-k)!}$ possible values for $(\sigma_1,
\ldots, \sigma_k)$, and integrating out the sum $T_{n-k}$, we arrive at
(\ref{eqn:joint.xnk}). Equation (\ref{eqn:sbp.joint.e}) follows by rewriting.
\end{pf}

%
Note that $X_n[k] = T_{n-k+1} - T_{n-k}$ for $k = 1, \ldots, n-1$. Thus
we can rewrite (\ref{eqn:joint.xnk}) in terms of the joint law of
$(T_{n}, T_{n-1}, \ldots, T_{n-k})$:
%
\begin{eqnarray}
\label{eqn:joint.T} &&\P(T_n \in \mathrm{d}t_0, \ldots, T_{n-k}
\in \mathrm{d}t_k)
\nonumber
\\[-8pt]
\\[-8pt]
\nonumber
&&\quad= \frac{n!}{(n-k)!} \Biggl(\prod
_{i=0}^{k-1} \frac{t_i - t_{i+1}}{t_i}\nu_1(t_i
- t_{i+1}) \Biggr) \nu_{n-k}(t_k) \,\mathrm{d}t_0
\cdots \mathrm{d}t_k.
\end{eqnarray}
Rearranging (\ref{eqn:joint.T}) yields the following result, which
appeared as an exercise in \cite{cyor}, Section~2.3.

%
\begin{cor}[(Chaumont--Yor \cite{cyor})]
The sequence $(T_n, T_{n-1},\ldots, T_1)$ is an inhomogeneous Markov
chain with transition probability
%
\begin{equation}
\label{eqn:T.tp} \P(T_{n-k} \in \mathrm{d}s| T_{n-k+1} = t) = (n-k+1)
\frac{t-s}{t}\nu _1(t-s) \frac{\nu_{n-k}(s)}{\nu_{n-k+1}(t)} \,\mathrm{d}s,
\end{equation}
for $k = 1, \ldots, n-1$. Together with $T_n \stackrel{d}{=} S_n$,
equation (\ref{eqn:T.tp}) specifies the joint law in (\ref
{eqn:joint.T}), and vice versa.
\end{cor}

\subsection{The stick-breaking representation}
An equivalent way to state (\ref{eqn:T.tp}) is that for $k \geq1$,
conditioned on $T_{n-k+1} = t$, $X_n[k]$ is distributed as the first
size-biased pick out of $n-k+1$ i.i.d. random variables conditioned to
have sum $S_{n-k+1} = t$. This provides a recursive way to generate a
finite i.i.d. size-biased permutation: first generate $T_n$ (which is
distributed as $S_n$). Conditioned on the value of $T_n$, generate
$T_{n-1}$ via (\ref{eqn:T.tp}), let $X_n[1]$ be the difference. Now
conditioned on the value of $T_{n-1}$, generate $T_{n-2}$ via (\ref
{eqn:T.tp}), let $X_n[2]$ be the difference, and so on. Let us explore
this recursion from a different angle by considering the ratio $
W_{n,k}:= \frac{X_n[k]}{T_{n-k+1}}$ and its complement, $\overline
{W}_{n,k} = 1 - W_{n,k} = \frac{T_{n-k}}{T_{n-k+1}}$. For $k \geq
2$, note that
%
\begin{equation}
\label{eqn:stickbreak} \frac{X_n[k]}{T_n} = \frac{X_n[k]}{T_{n-k+1}} \frac
{T_{n-k+1}}{T_{n-k+2}}\cdots
\frac{T_{n-1}}{T_n} = W_{n,k}\prod_{i =
1}^{k-1}
\overline{W}_{n,i}.
\end{equation}
The variables $\overline{W}_{n,i}$ can be interpreted as residual
fractions in a \emph{stick-breaking} scheme: start with a stick of
length 1. Choose a point on the stick according to distribution
$W_{n,1}$, `break' the stick into two pieces, discard the piece of
length $W_{n,1}$ and rescale the remaining half to have length~1.
Repeating this procedure $k$ times, and (\ref{eqn:stickbreak}) is the
fraction broken off at step $k$ relative to the original stick length.

Together with $T_n \stackrel{d}{=} S_n$, one could use (\ref
{eqn:stickbreak}) to compute the marginal distribution for $X_n[k]$ in
terms of the ratios $\overline{W}_{n,i}$. In general the $W_{n,i}$ are
not necessarily independent, and their joint distributions need to be
worked out from (\ref{eqn:T.tp}). However, when $F$ has gamma
distribution, $T_n, W_{n,1}, \ldots, W_{n,k}$ are independent, and
(\ref
{eqn:stickbreak}) leads to the following result of Patil and Taillie
\cite{pt77}.

\begin{prop}[(Patil--Taillie \cite{pt77})]\label{prop:pt} If $F$ has
distribution $\operatorname{gamma}(a, \lambda)$ for some $a, \lambda> 0$, then $T_n$
and the $W_{n,1}, \ldots, W_{n,n-1}$ in (\ref{eqn:stickbreak}) are
mutually independent. In this case,
\begin{eqnarray*}
X_n[1] &=& \gamma_0 \beta_1,
\\
X_n[2] &=& \gamma_0 \bar{\beta}_1,
\beta_2,
\\
& \cdots&
\\
X_n[n-1] &=& \gamma_0 \bar{\beta}_1 \bar
\beta_2\cdots\bar\beta _{n-2}\beta_{n-1},
\\
X_n[n] &=& \gamma_0 \bar{\beta}_1 \bar
\beta_2\cdots\bar\beta_{n-1},
\end{eqnarray*}
where $\gamma_0$ has distribution $\operatorname{gamma}(an,\lambda)$, $\beta_k$ has
distribution $\operatorname{beta}(a+1, (n-k)a)$, $\bar\beta_k = 1 - \beta_k$ for $1
\leq k \leq n-1$, and the random variables $\gamma_0, \beta_1, \ldots,
\beta_{n-1}$ are independent.
\end{prop}

\begin{pf} This statement appeared as a casual in-line statement
without proof in \cite{pt77}, perhaps since there is an elementary
proof (which we will outline later). For the sake of demonstrating
previous computations, we shall start with (\ref{eqn:T.tp}). By
assumption, $S_k$ has distribution $\operatorname{gamma}(ak, \lambda)$. One
substitutes the density of $\operatorname{gamma}(ak, \lambda)$ for $\nu_k$ to obtain
\begin{eqnarray*}
\P(T_{n-k} \in \mathrm{d}s | T_{n-k+1} = t) &=& 
C
\biggl(\frac{(t-s)^a}{t} \biggr) \biggl(\frac
{s^{a(n-k)-1}}{t^{a(n-k)+a-1}} \biggr)
\\
&=& 
C \biggl(1 - \frac{s}{t}
\biggr)^a \biggl(\frac{s}{t} \biggr)^{a(n-k)-1}
\end{eqnarray*}
for some normalizing constant $C$. By rearranging, we see that $
\frac
{T_{n-k}}{T_{n-k+1}}$ has distribution $\operatorname{beta}(a+1,a(n-k))$, and is
independent of $T_{n-k+1}$. Therefore $W_{n,1}$ is independent of
$T_n$. By the stick-breaking construction, $W_{n,2}$ is independent of
$T_{n-1}$ and $T_n$, and hence of $W_{n,1}$. The final formula follow
from rearranging (\ref{eqn:stickbreak}).
\end{pf}

Here is another direct proof. By the stick-breaking construction, it is
sufficient to show that $T_n$ is independent of $ W_{n,1} = \frac
{X_n[1]}{T_n}$. Note that
%
\begin{eqnarray}
\label{eqn:sb.lukacs}&& \P\bigl(X_n[1]/T_n \in \mathrm{d}u, T_n
\in \mathrm{d}t\bigr)
\nonumber
\\[-8pt]
\\[-8pt]
\nonumber
&&\quad = n u \P \biggl(\frac
{X_n(1)}{X_n(1) + (X_n(2)+\cdots+X_n(n))} \in \mathrm{d}u, T_n \in \mathrm{d}t
\biggr).
\end{eqnarray}
Since $X_n(1) \stackrel{d}{=} \operatorname{gamma}(a,1)$, $S_{n-1} = X_n(2) + \cdots+
X_n(n) \stackrel{d}{=} \operatorname{gamma}(a(n-1),1)$, independent of $X_n(1)$, the
ratio $\frac{X_n(1)}{X_n(1) + S_{n-1}}$ has distribution $\operatorname{beta}(a,
a(n-1))$ and is independent of $T_n$. Thus,
\begin{eqnarray*}
\P\bigl(X_n[1]/T_n \in \mathrm{d}u\bigr) &=& n u
\frac{\Gamma(a+a(n-1))}{\Gamma(a)\Gamma
(a(n-1))}u^{a-1}(1-u)^{a(n-1)-1}
\\
&=& \frac{\Gamma(a+1+a(n-1))}{\Gamma(a+1)\Gamma
(a(n-1))}u^{a}(1-u)^{a(n-1)-1}.
\end{eqnarray*}
In other words, $X_n[1]/T_n \stackrel{d}{=} \operatorname{beta}(a, a(n-1))$. This
proves the claim.

Lukacs \cite{lukacs} proved that if $X, Y$ are non-degenerate, positive
independent random variables, then $X+Y$ is independent of $\frac
{X}{X+Y}$ if and only if both $X$ and $Y$ have gamma distributions with
the same scale parameter. Thus, one obtains another characterization of
the gamma distribution.

\begin{cor}[(Patil--Taillie converse)]\label{cor:pt} If $T_n$ is
independent of $X_n[1]/T_n$, then $F$ is $\operatorname{gamma}(a, \lambda)$ for some
$a, \lambda> 0$.
\end{cor}

\begin{pf} One applies Lukacs' theorem to $X_n(1)$ and $(X_n(2) +
\cdots+ X_n(n))$ in (\ref{eqn:sb.lukacs}).
\end{pf}

\section{Size-biased permutation as induced order statistics}\label{sec:ios}
When $n$ i.i.d. pairs $(X_n(i), Y_n(i))$ are ordered by their
$Y$-values, the corresponding $X_n(i)$ are called the \emph{induced
order statistics} of the vector $Y_n$, or its \emph{concomitants}.
Gordon \cite{gordon83} first proved the following result for finite $n$
which shows that finite i.i.d. size-biased permutation is a form of
induced order statistics. Here we state the infinite sequence version,
which is a special case of~\cite{pitmanyor}, Lemma~4.4.

\begin{prop}[(Perman, Pitman and Yor \cite{pitmanyor})]\label
{prop:coupling} Let $x$ be a fixed positive sequence with finite sum $t
= \sum_{i=1}^\infty x(i)$, $\varepsilon$ a sequence of i.i.d. standard
exponential random variables, independent of $x$. Let $Y$ be the
sequence with $Y(i) = \varepsilon(i)/x(i)$, $i = 1, 2, \ldots,
Y^\uparrow
$ its sequence of increasing order statistics. Define $X^\ast(k)$ to be
the value of the $x(i)$ such that $Y(i)$ is $Y^{\uparrow}(k)$. Then
$(X^\ast(k), k = 1, 2, \ldots)$ is a size-biased permutation of the
sequence $x$. In particular, the size-biased permutation of a positive
i.i.d. sequence $(X_n(1), \ldots, X_n(n))$ is distributed as the
induced order statistics of the sequence $(Y_n(i) = \varepsilon
_n(i)/{X_n(i)}, 1 \leq i \leq n)$ for an independent sequence of i.i.d.
standard exponentials $(\varepsilon_n(1), \ldots, \varepsilon_n(n))$,
independent of the $X_n(i)$.
\end{prop}

\begin{pf} Note that the $Y(i)$ are independent exponentials with
rates $x(i)$. Let $\sigma$ be the random permutation such $Y(\sigma(i))
= Y^\uparrow(i)$. Note that $X^\ast(k) = x(\sigma(k))$. Then
\[
\P\bigl(\sigma(1) = i\bigr) = \mathbb{P}\bigl(Y(i) = \min\bigl\{Y(j), j = 1, 2,
\ldots \bigr\}\bigr) = \frac{x(i)}{t},
\]
thus $X^\ast(1) \stackrel{d}{=} x[1]$. In general, for distinct indices
$i_1, \ldots, i_k$, by the memoryless property of the exponential distribution,
\begin{eqnarray*}
&& \P\bigl(\sigma(k) = i_k | \sigma(1) = i_1, \ldots,
\sigma(k) = i_{k-1} \bigr)
\\
&&\quad= \mathbb{P} \bigl(Y(i_k) = \min\bigl\{Y\bigl(\sigma(j)\bigr), j
\geq k \bigr\} | \sigma (1) = i_1, \ldots, \sigma(k) =
i_{k-1} \bigr) = \frac{x(i_k)}{t -
\sum_{j=1}^{k-1} x(i_j)}.
\end{eqnarray*}
Induction on $k$ completes the proof.
\end{pf}

Proposition~\ref{prop:coupling} readily supplies simple proofs for
joint, marginal and asymptotic distributions of i.i.d. size-biased
permutation. For instance, the proof of the following nesting property,
which can be cumbersome, amounts to i.i.d. thinning.

%
\begin{cor}Consider a finite i.i.d. size-biased permutation $(X_n[1],
\ldots, X_n[n])$ from a distribution $F$. For $1 \leq m \leq n$, select
$m$ integers $a_1 < \cdots< a_m$ by uniform sampling from $\{1, \ldots
, n\}$ without replacement. Then the subsequence $\{X_n[a_j], 1 \leq j
\leq m\}$ is jointly distributed as a finite i.i.d. size-biased
permutation of length  $m$ from $F$.
\end{cor}

In general, the induced order statistics representation of size-biased
permutation is often useful in studying limiting distribution as $n \to
\infty$, since one can consider the i.i.d. pair $(X_n(i), Y_n(i))$ and
appeal to tools from empirical process theory. We shall demonstrate
this in Sections \ref{sec:asymp} and \ref{sec:ppp}.

\subsection{Joint and marginal distribution revisited} \label
{subsec:ios.joint}

We now revisit the results in Section~\ref{subsec:joint.ppy} using
induced order statistics. This leads to a different formula for the
joint distribution, and an alternative proof of the Barouch--Kaufman
formula (\ref{eqn:sbp.joint.e}).

%
\begin{prop}\label{lem:joint.x.u}
$(X_n[k], k = 1, \ldots, n)$ is distributed as the first coordinate of
the sequence of pairs $((X^{\ast}_n(k), U_n^\downarrow(k)), k =1,
\ldots
, n)$, where $U_n^\downarrow(1) \geq\cdots\geq U^\downarrow_n(n)$ is
a sequence of uniform order statistics, and conditional on
$(U^\downarrow_n(k) = u_k, 1 \leq k \leq n)$, the $X^{\ast}_n(k)$ are
independent with distribution $(G_{u_k}(\cdot), k = 1, \ldots, n)$, where
%
\begin{equation}
\label{eqn:Gu} G_u(\mathrm{d}x) = \frac{x\mathrm{e}^{-\phi^{-1}(u)x} F(\mathrm{d}x)}{-\phi'(\phi^{-1}(u))}.
\end{equation}
Here $\phi$ is the Laplace transform of $X$, that is, $\phi(y) =
\int_0^\infty \mathrm{e}^{-yx} F(\mathrm{d}x)$, $\phi'$ its derivative and $\phi^{-1}$ its
inverse function.
\end{prop}

\begin{pf} Let $X_n$ be the sequence of $n$ i.i.d. draws from $F$,
$\varepsilon_n$ an independent sequence of i.i.d. standard exponentials,
$Y_n(i) = \varepsilon_n(i)/X_n(i)$ for $i = 1, \ldots, n$. Note that the
pairs $\{(X_n(i), Y_n(i)), 1 \leq i \leq n\}$ is an i.i.d. sample from
the joint distribution $ F(\mathrm{d}x)[x\mathrm{e}^{-yx} \,\mathrm{d}y]$. Thus, $Y_n(i)$ has
marginal density
%
\begin{equation}
P\bigl(Y_n(i) \in \mathrm{d}y\bigr) = -\phi'(y) \,\mathrm{d}y,\qquad 0 < y <
\infty, \label{1.b}
\end{equation}
and its distribution function is $F_Y = 1 - \phi$. Given $\{Y_n(i) =
y_i, 1 \leq i \leq n\}$, the $X_n^\ast(i)$ defined in Proposition~\ref
{prop:coupling} are independent with conditional distribution
$\widetilde{G}(y_i, \cdot)$ where
%
\begin{equation}
\widetilde{G}(y, \mathrm{d}x) = \frac{x\mathrm{e}^{-yx} F(\mathrm{d}x)}{-\phi'(y)}. \label{eqn:Gy}
\end{equation}
Equation (\ref{eqn:Gu}) follows from writing the order statistics as
the inverse transforms of ordered uniform variables
%
\begin{eqnarray}
\label{eqn:ynun} \bigl(Y_n^\uparrow(1), \ldots,
Y^\uparrow_n(n)\bigr) &\stackrel{d} {=}& \bigl(F_Y^{-1}
\bigl(U^\downarrow_n(n)\bigr), \ldots, F_Y^{-1}
\bigl(U^\downarrow_n(1)\bigr)\bigr)
\nonumber
\\[-8pt]
\\[-8pt]
\nonumber
&\stackrel{d} {=} &\bigl(
\phi^{-1}\bigl(U^\downarrow_n(1)\bigr), \ldots, \phi
^{-1}\bigl(U^\downarrow_n(n)\bigr)\bigr),
\end{eqnarray}
where $(U^\downarrow_n(k), k = 1, \ldots, n)$ is an independent
decreasing sequence of uniform order statistics. Note that the minus
sign in (\ref{1.b}) results in the reversal of the sequence $U_n$ in
the second equality of (\ref{eqn:ynun}).
\end{pf}

\begin{cor} \label{cor:marginal}
For $1 \leq k \leq n$ and $0 < u < 1$, let
%
\begin{equation}
\label{2.a} f_{n,k}(u) = \frac{\mathbb{P}(U^\uparrow_{n}(k) \in \mathrm{d}u)}{\mathrm{d}u} = n\pmatrix {n-1
\cr
k-1}u^{n-k}(1-u)^{k-1}
\end{equation}
be the density of the $k$th largest of the $n$ uniform order statistics
$(U^\uparrow_n(i), i = 1, \ldots, n)$. Then
%
\begin{equation}
\label{2.c} \frac{\mathbb{P}(X_n[k] \in \mathrm{d}x)}{xF(\mathrm{d}x)} = \int_0^\infty
\mathrm{e}^{-xy}f_{n,k}\bigl(\phi(y)\bigr) \,\mathrm{d}y.
\end{equation}
\end{cor}

\begin{pf}
Equation (\ref{2.a}) follows from known results on order statistics,
see \cite{evtHaan}. For $u \in[0,1]$, let $y = \phi^{-1}(u)$. Then
$\frac{\mathrm{d}y}{\mathrm{d}u} = \frac{1}{\phi'(\phi^{-1}(u))}$ by the inverse
function theorem. Apply this change of variable to (\ref{eqn:Gu}),
rearrange and integrate with respect to $y$ to obtain (\ref{2.c}).
\end{pf}

In particular, for the first and last values,
\begin{eqnarray*}
\frac{\mathbb{P}(X_n[1] \in \mathrm{d}x)}{xF(\mathrm{d}x)} &=& n\int_0^\infty
\mathrm{e}^{-xy}\phi (y)^{n-1} \,\mathrm{d}y,
\\
\frac{\mathbb{P}(X_n[n] \in \mathrm{d}x)}{xF(\mathrm{d}x)} &=& n\int_0^\infty
\mathrm{e}^{-xy}\bigl(1-\phi(y)\bigr)^{n-1} \,\mathrm{d}y.
\end{eqnarray*}

\subsubsection{Alternative derivation of the Barouch--Kaufman formula}
Write $\phi(y) = \E(\mathrm{e}^{-yX})$ for $X$ with distribution $F$. Then
$\phi
(y)^{n-1} = \E(\mathrm{e}^{-yS_{n-1}})$ where $S_{n-1}$ is the sum of $(n-1)$
i.i.d. random variables with distribution $F$. Since all integrals
involved are finite, by Fubini's theorem
\begin{eqnarray*}
\frac{\mathbb{P}(X_n[1] \in \mathrm{d}x)}{xF(\mathrm{d}x)} &=& n\int_0^\infty
\mathrm{e}^{-xy}\phi (y)^{n-1} \,\mathrm{d}y = n \E \biggl(\int
_0^\infty \mathrm{e}^{-xy}\mathrm{e}^{-yS_{n-1}} \,\mathrm{d}y
\biggr)
\\
&= &n \E \biggl(\frac{1}{x + S_{n-1}} \biggr),
\end{eqnarray*}
which is a rearranged version of the Barouch--Kaufman formula (\ref
{eqn:sbp.joint.e}) for $k = 1$. Indeed, one can derive the entire
formula from Proposition~\ref{lem:joint.x.u}. For simplicity, we
demonstrate the case $k = 2$.

\begin{pf*}{Proof of (\ref{eqn:sbp.joint.e}) for $\bolds{k = 2}$}
The joint distribution of the two largest uniform order statistics
$U_n^{\downarrow}(1), U_n^{\downarrow}(2)$ has density
\[
f(u_1, u_2) = n(n-1)u_2^{n-1}\qquad
\mbox{for } 0 \leq u_2 \leq u_1 \leq1.
\]
Conditioned on $U_n^{\downarrow}(1) = u_1, U_n^{\downarrow}(2) = u_2$,
$X_n[1]$ and $X_n[2]$ are independent with distribution~(\ref{eqn:Gu}).
Let $y_1 = \phi^{-1}(u_1), y_2 = \phi^{-1}(u_2)$, so $ \mathrm{d}y_1 =
\frac
{\mathrm{d}u_1}{\phi'(\phi^{-1}(u_1))}, \mathrm{d}y_2 = \frac{\mathrm{d}u_2}{\phi'(\phi
^{-1}(u_2))}$. Let $S_{n-2}$ denote the sum of $(n-2)$ i.i.d. random
variables with distribution $F$. Apply this change of variable and
integrate out $y_1, y_2$, we have
\begin{eqnarray*}
\frac{\P(X_n[1] \in \mathrm{d}x_1, X_n[2]
\in \mathrm{d}x_2)}{x_1x_2F(\mathrm{d}x_1)F(\mathrm{d}x_2)}& =& n(n-1)\int_0^\infty\int
_{y_1}^\infty \mathrm{e}^{-y_1x_1}\mathrm{e}^{-y_2x_2}\bigl(
\phi (y_2)\bigr)^{n-2} \,\mathrm{d}y_2\,\mathrm{d}y_1
\\
&=& n(n-1)\int_0^\infty\int_{y_1}^\infty
\mathrm{e}^{-y_1x_1}\mathrm{e}^{-y_2x_2}\E \bigl(\mathrm{e}^{-y_2S_{n-2}}\bigr)
\,\mathrm{d}y_2\,\mathrm{d}y_1
\\
&=& n(n-1)\E \biggl(\int_0^\infty\int
_{y_1}^\infty \mathrm{e}^{-y_1x_1}\mathrm{e}^{-y_2(x_2+S_{n-2})}
\,\mathrm{d}y_2\,\mathrm{d}y_1 \biggr)
\\
&=& n(n-1)\E \biggl(\int_0^\infty \mathrm{e}^{-y_1x_1}
\frac
{\mathrm{e}^{-y_1(x_2+S_{n-2})}}{x_2 + S_{n-2}} \,\mathrm{d}y_1 \biggr)
\\
&=& n(n-1)\E \biggl(\frac{1}{(x_2+S_{n-2})(x_1 + x_2 + S_{n-2})} \biggr),
\end{eqnarray*}
where the swapping of integrals is justified by Fubini's theorem, since
all integrals involved are finite.
\end{pf*}

%

\begin{ex}\label{ex:gamma1}
Suppose $F$ is $\operatorname{gamma}(a,1)$. Then $\phi(y) = (\frac{1}{1+y})^a$, and
$\phi^{-1}(u) =  u^{-1/a} - 1$.
Hence $G_u$ in (\ref{eqn:Gu}) is
\[
G_u(\mathrm{d}x) = \frac{x}{au^{(a+1)/a}} \mathrm{e}^{-(u^{-1/a} - 1)x} F(\mathrm{d}x) =
\frac
{x^a}{\Gamma(a+1)}u^{-(a+1)/a}\mathrm{e}^{-xu^{-1/a}}.
\]
That is, $G_u$ is $\operatorname{gamma}(a+1, u^{-1/a})$.
\end{ex}

\subsubsection{Patil--Taillie revisited}
When $F$ is $\operatorname{gamma}(a,\lambda)$, Lemma~\ref{lem:joint.x.u} gives the
following result, which is an interesting complement to the
Patil--Taillie representation in Proposition~\ref{prop:pt}.

\begin{prop}\label{prop:newgb}
Suppose $F$ is $\operatorname{gamma}(a,\lambda)$. Then $G_u$ is $\operatorname{gamma}(a+1, \lambda
u^{-1/a})$, and
%
\begin{equation}
\label{eqn:xrev.g} \bigl(X_n[k], k = 1, \ldots, n\bigr) \stackrel{d} {=}
\bigl(\bigl[U^\downarrow _n(k)\bigr]^{1/a}
\gamma_k, k = 1, \ldots, n\bigr),
\end{equation}
where $\gamma_1, \ldots, \gamma_n$ are i.i.d. $\operatorname{gamma}(a+1, \lambda)$
random variables, independent of the sequence of decreasing uniform
order statistics $(U^\downarrow_n(1), \ldots, U^\downarrow_n(n))$.
Alternatively, jointly for $k =  1, \ldots, n$
\begin{eqnarray*}
X_n^{\mathrm{rev}}[1] &=& \gamma_1 \beta_{an, 1},
\\
X_n^{\mathrm{rev}}[2] &=& \gamma_2 \beta_{an, 1},
\beta_{an-a, 1},
\\
& \cdots&
\\
X_n^{\mathrm{rev}}[n-1] &=& \gamma_{n-1}
\beta_{an, 1}\beta_{an-a, 1}\cdots \beta _{2a, 1},
\\
X_n^{\mathrm{rev}}[n] &=& \gamma_{n} \beta_{an, 1}
\beta_{an-a, 1}\cdots\beta _{a, 1},
\end{eqnarray*}
where the $\beta_{an - ia,1}$ for $i = 0, \ldots, n-1$ are distributed
as $\operatorname{beta}(an-ia, 1)$, and they are independent of each other and the
$\gamma_k$.
\end{prop}

\begin{pf} The distribution $G_u$ is computed in the same way as in
Example~\ref{ex:gamma1} and (\ref{eqn:xrev.g}) follows readily from
Proposition~\ref{lem:joint.x.u}.
\end{pf}

A direct comparison of the two different representations in Propositions
\ref{prop:pt} and \ref{prop:newgb} creates $n$ distributional
identities. For example, the equality $X_n[1] = X^{\mathrm{rev}}_n[n]$ shows
that the following two means of creating a product of independent
random variables produce the same result in law:
%
\begin{equation}
\label{eqn:beta.ini} \beta_{a+1, (n-1)a} \gamma_{an, \lambda} \stackrel{d} {=}
\beta_{an,
1} \gamma_{a+1, \lambda},
\end{equation}
where $\gamma_{r,\lambda}$ and $\beta_{a,b}$ denote random variables
with distributions $\operatorname{gamma}(r, \lambda)$ and $\operatorname{beta}(a,b)$, respectively.
Indeed, this identity comes from the usual `beta--gamma' algebra, which
allows us to write
\[
\beta_{a+1, (n-1)a} = \frac{\gamma_{a+1,\lambda}}{\gamma
_{a+1,\lambda}
+ \gamma_{a(n-1),\lambda}}, \qquad\beta_{an,1} =
\frac
{\gamma
_{an,\lambda}}{\gamma_{an,\lambda} + \gamma_{1,\lambda}}
\]
for $\gamma_{a(n-1),\lambda}, \gamma_{1,\lambda}$ independent of all
others. Thus, (\ref{eqn:beta.ini}) reduces to
\[
\gamma_{a+1,\lambda} + \gamma_{a(n-1),\lambda} \stackrel{d} {=} \gamma
_{an,\lambda} + \gamma_{1,\lambda},
\]
which is true since both sides have distribution $\operatorname{gamma}(an+1, \lambda)$.

\section{Limit in distributions of finite size-biased
permutations}\label{sec:sbp}
As hinted in the \hyperref[sec1]{Introduction}, our setup is a finite version of the
size-biased permutation of ranked jumps of a subordinator, studied in
\cite{pitmanyor}. In this section, we make this statement rigorous (see
Proposition~\ref{prop:converge}).

Let $\Delta= \{x = (x(1), x(2), \ldots)\dvt  x(i) \geq0, \sum_ix(i)
\leq
1\}$ and $\Delta^\downarrow= \{x^\downarrow\dvt  x \in\Delta\}$ be closed
infinite simplices, the later contains sequences with non-increasing
terms. Denote their boundaries by $\Delta_1 = \{x \in\Delta\dvt  \sum_ix(i) = 1\}$ and $\Delta_1 ^\downarrow= \{x \in\Delta^\downarrow,
\sum_ix(i) = 1\}$, respectively. Any finite sequence can be associated
with an element of $\Delta_1$ after being normalized by its sum and
extended with zeros. Thus, one can speak of convergence in distribution
of sequences in $\Delta$.

We have to consider $\Delta$ and not just $\Delta_1$ because a sequence
in $\Delta_1$ can converge to one in $\Delta$. For example, the
sequence $(X_n, n \geq1) \in\Delta_1$ with $X_n(i) = 1/n$ for all $i
= 1, \ldots, n$ converges to the elementwise zero sequence in $\Delta$.
Thus, we need to define convergence in distribution of size-biased
permutations in $\Delta$. We shall do this using Kingman's paintbox. In
particular, with this definition, convergence of size-biased
permutation is equivalent to convergence of order statistics. Our
treatment in Section~\ref{subsec:paintbox} follows that of Gnedin
\cite
{gnedin} with simplified assumptions. The proofs can be found in \cite{gnedin}.

It then follows that size-biased permutation of finite i.i.d. sequences
with almost sure finite sum converges to the size-biased permutation of
the sequence of ranked jumps of a \emph{subordinator}, roughly
speaking, a non-decreasing process with independent and homogeneous
increments. We give a review of L\'{e}vy processes and subordinators in
Section~\ref{subsec:levy}. Many properties such as stick-breaking and
the Markov property of the remaining sum have analogues in the limit.
We explore these in Section~\ref{subsec:parallel}.

\subsection{Kingman's paintbox and some convergence theorems}\label
{subsec:paintbox}
\emph{Kingman's paintbox} \cite{ki78a} is a useful way to describe and
extend size-biased permutations. For $x \in\Delta$, let $s_k$ be the
sum of the first $k$ terms. Note that $x$ defines a partition $\varphi
(x)$ of the unit interval $[0,1]$,
consisting of \emph{components} which are intervals of the form
$[s_{k}, s_{k+1})$ for $k = 1, 2, \ldots,$ and the interval $[s_\infty,
1]$, which we call the \emph{zero component}. Sample points $\xi_1,
\xi
_2, \ldots$ one by one from the uniform distribution on $[0,1]$. Each
time a sample point discovers a new component that is not in $[s_\infty
, 1]$, write down its size. If the sample point discovers a new point
of $[s_\infty, 1]$, write $0$. Let $X^\ast= (X^\ast(1), X^\ast(2),
\ldots)$ be the random sequence of sizes. Since the probability of
discovery of a particular (non-zero) component is proportional to its
length, the non-zero terms in $X^\ast$ form the size-biased permutation
of the non-zero terms in $x$ as defined by (\ref{eqn:sbp}).
In Kingman's paintbox terminology, the components correspond to
different colors used to paint the balls with labels $1, 2, \ldots.$
Two balls $i, j$ have the same paint color if and only if $\xi_i$ and
$\xi_j$ fall in the same component. The size-biased permutation
$X^\ast
$ records the size of the newly discovered components, or paint colors.
The zero component represents a continuum of distinct paint colors,
each of which can be represented at most once.

By construction, the size-biased permutation of a sequence $x$ does not
depend on the ordering of its terms. In particular, convergence in
$\Delta^\downarrow$ implies convergence in distribution of the
corresponding sequences of size-biased permutations. The converse is
also true. Proofs of the following statements can be found in \cite{gnedin}.

\begin{thmm}[(Equivalence of convergence of order statistics and
s.b.p. \cite{gnedin})]\label{thm:sb.orderstatsbp}
Suppose $X^\downarrow$, $X^\downarrow_1, X^\downarrow_2, \ldots$ are
random elements of $\Delta^\downarrow$ and $X^\downarrow_n \stackrel
{{f.d.d.}}{\rightarrow} X^\downarrow$. Then $(X_n^\downarrow)^\ast
\stackrel{{f.d.d.}}{\rightarrow} (X^\downarrow)^\ast$.

Conversely, suppose $X_1, X_2, \ldots$ are random elements of $\Delta$
and $X^\ast_n \stackrel{{f.d.d.}}{\rightarrow} Y$ for some $Y \in
\Delta
$. Then $Y^\ast\stackrel{d}{=} Y$, and $X^\downarrow_n \stackrel
{{f.d.d.}}{\rightarrow} Y^\downarrow$.
\end{thmm}

We can speak of convergence of size-biased permutation on $\Delta$
without having to pass to order statistics. However, convergence of
random elements in $\Delta$ implies neither convergence of order
statistics nor of their size-biased permutations. To achieve
convergence, we need to keep track of the sum of components. This
prevents large order statistics from `drifting' to infinity.

%
\begin{thmm}[(\cite{gnedin})]\label{thm:sb.upgrade}
Suppose $X, X_1, X_2, \ldots$ are random elements of $\Delta$ and
\[
\biggl(X_n, \sum_iX_n(i)
\biggr) \stackrel{{f.d.d.}} {\rightarrow} \biggl(X, \sum
_iX(i) \biggr).
\]
Then $X^\ast_n \stackrel{{f.d.d.}}{\rightarrow} X^\ast$ and
$X^\downarrow
_n \stackrel{{f.d.d.}}{\rightarrow} X^\downarrow$.
\end{thmm}

\subsection{Finite i.i.d. size-biased permutation and ranked jumps of a~subordinator}\label{subsec:levy}
Let $((X_n), n \geq1)$ be an i.i.d. positive triangular array, that
is, $X_n = (X_n(1), \ldots, X_n(n))$, where $X_n(i), i = 1, \ldots, n$
are i.i.d. and a.s. positive. Write $T_n$ for $\sum_{i=1}^nX_n(i)$.
We ask for conditions under which the size-biased permutation of the
sequence $(X_n, n \geq1)$ converges to the size-biased permutation of
some infinite sequence $X$. Let us restrict to the case $T_n \stackrel
{d}{\rightarrow} T$ for some $T < \infty$ a.s. A classical result in
probability states that $T_n \stackrel{d}{\rightarrow} T$ if and only
if $T = \tilde{T}(1)$ for some \emph{L\'evy process} $\tilde{T}$, which
in this case is a \emph{subordinator}.
For self-containment, we gather some necessary facts about L\'evy
processes and subordinators below. See \cite{kallenberg}, Section~15, for
their proofs, \cite{bertoin99}~for a thorough treatment of subordinators.

\begin{defn} A \emph{L\'evy process} $\tilde{T}$ in $\R$ is a
stochastic process with right-continuous left-limits paths, stationary
independent increments, and $\tilde{T}(0) = 0$. A \emph{subordinator}
$\tilde{T}$ is a L\'evy process, with real, finite, non-negative increments.
\end{defn}

%
Following \cite{kallenberg}, we do not allow the increments to have
infinite value. We suffer no loss of generality, since subordinators
with jumps of possibility infinite size do not contribute to our
discussion of size-biased permutation. Let $\tilde{T}$ be a
subordinator, $T = \tilde{T}(1)$. For $t, \lambda\geq0$, using the
fact that increments are stationary and independent, one can show that
\[
\E\bigl(\exp\bigl(-\lambda\tilde{T}(t)\bigr)\bigr) = \exp\bigl(-t\Phi(\lambda)
\bigr),
\]
where the function $\Phi\dvtx [0, \infty) \to[0, \infty)$ is called the
\emph{Laplace exponent} of $\tilde{T}$. It satisfies the \emph{L\'
evy--Khinchine formula}
\[
\Phi(\lambda) = \mathsf{d} \lambda+ \int_0^\infty
\bigl(1 - \mathrm{e}^{-\lambda x}\bigr) \Lambda(\mathrm{d}x),\qquad \lambda\geq0,
\]
where $\mathsf{d} > 0$ is the \emph{drift coefficient}, and $\Lambda
$ a
unique measure on $(0, \infty)$ with $\Lambda([1, \infty)) < \infty$,
called the \emph{L\'evy measure of $\tilde{T}$}. Assume $\int_0^1x
\Lambda(\mathrm{d}x) < \infty$, which implies a.s. $\tilde{T}(1) = T < \infty$.
Then over $[0,1]$, $\tilde{T}$ is the sum of a deterministic drift plus
a Poisson point process (p.p.p.) with i.i.d. jumps
\[
(\tilde T) (t) = \mathsf{d} t + \sum_iX(i)
\mathbf{1}_{\{\sigma(i)
\leq
t\}}
\]
for $0 \leq t \leq1$, where $\{(\sigma(i), X(i)), i \geq1\}$ are
points in a p.p.p. on $(0, \infty)^2$ with intensity measure $\mathrm{d}t
\Lambda(\mathrm{d}x)$. The $X(i)$ are the \emph{jumps} of $\tilde{T}$.
Finally, we need a classical result on convergence of i.i.d. positive
triangular arrays to subordinators (see \cite{kallenberg}, Section~15).

%
\begin{thmm}\label{thm:sb.classic} Let $(X(n), n \geq1)$ be an i.i.d.
positive triangular array, $T_n = \sum_{i=1}^nX_n(i)$. Then $T_n
\stackrel{d}{\rightarrow} T$ for some random variable $T$, $T <
\infty$
a.s. if and only if $T = \tilde{T}(1)$ for some subordinator $\tilde
{T}$ whose L\'evy measure $\Lambda$ satisfies $\int_0^1x \Lambda(\mathrm{d}x)
< \infty$. Furthermore, let $\mu_n$ be the measure of $X_n(i)$. Then on
$\R_+$, the sequence of measures $(n\mu_n)$ converges vaguely to
$\Lambda$, written
\[
n\mu_n \stackrel{v} {\rightarrow} \Lambda.
\]
That is, for all $f\dvtx  \R_+ \to\R_+$ continuous with compact support,
$n\mu_n(f) = n\int_0^\infty f(x) \mu_n(\mathrm{d}x)$ converges to $\Lambda(f)
= \int_0^\infty f(x) \Lambda(\mathrm{d}x)$. In particular, if $\mu_n, \Lambda$
have densities $\rho_n, \rho$, respectively, then we have pointwise
convergence for all $x > 0$
\[
n\rho_n(x) \to\rho(x).
\]
\end{thmm}

\begin{prop}\label{prop:converge}
Let $(X_n, n \geq1)$ be an i.i.d. positive triangular array, $T_n =
\sum_{i=1}^nX_n(i)$. Suppose $T_n \stackrel{d}{\rightarrow} T$ for some
$T$ a.s. finite. Let $X$ be the sequence of ranked jumps of $T$
arranged in any order, $(X/T)^\ast$ be the size-biased permutation of
the sequence $(X/T)$ as defined using Kingman's paintbox, $(X^\ast)' =
T \cdot(X/T)^\ast$. Then
\[
(X_n)^\ast\stackrel{{f.d.d.}} {\rightarrow}
\bigl(X^\ast\bigr)'.
\]
\end{prop}

\begin{pf}
The sequence of decreasing order statistics $X^\downarrow_n$ converges
in distribution to $X^\downarrow$ \cite{kallenberg}. Since $T_n, T > 0$
a.s. and $T_n \stackrel{d}{\rightarrow} T$, $X^\downarrow_n/T_n
\stackrel{\mathrm{f.d.d.}}{\rightarrow} X^\downarrow/T$. Theorem~\ref
{thm:sb.orderstatsbp} combined with multiplying through by $T$ prove
the claim.
\end{pf}

For subordinators without drift, $\mathsf{d} = 0$, $\sum_iX(i) = T$,
hence $(X^\ast)' = X^\ast$. When $\mathsf{d} > 0$, the sum of the jumps
$\sum_iX_i$ is strictly less than $T$, so $(X^\ast)' \neq X^\ast$. In
this case, there is a non-trivial zero component coming from an
accumulation of mass at $0$ of $n\mu_n$ in the limit as $n \to\infty$.
At each finite, large $n$, we have a significant number of jumps with
`microscopic' size.

The case without drift was studied by Perman, Pitman and Yor in \cite
{pitmanyor} with the assumption $\Lambda(0, \infty) = \infty$ to ensure
that the sequence of jumps has infinite length. We shall re-derive some
of their results as limits of results for finite i.i.d. size-biased
permutation using Theorem~\ref{thm:sb.classic} in the next section. One
can obtain another finite version of the Perman--Pitman--Yor setup by
letting $\Lambda(0, \infty) <  \infty$, but this can be reduced to
finite i.i.d. size-biased permutation by conditioning. Specifically,
$\tilde{T}$ is now a compound Poisson process, where the subordinator
waits for an exponential time with rate $\Lambda(0, \infty)$ before
making a jump $X$, whose length is independent of the waiting time and
distributed as $\P(X \leq t) = \Lambda(0, t]/\Lambda(0, \infty)$
\cite
{bertoin99}. If $(X(1), X(2), \ldots)$ is the sequence of successive
jumps of $(\tilde{T}_s, s \geq0)$, then $(X(1), X(2), \ldots, X(N))$
is the sequence of successive jumps of $(\tilde{T}_s, 0 \leq s \leq
1)$, where $N$ is a Poisson random variable with mean $\Lambda(0,
\infty
)$, independent of the jump sequence $(X(1), X(2), \ldots)$. For $N >
0$, properties of the size-biased permutation of $(X(1), \ldots, X(N))$
can be deduced from those of a finite i.i.d. size-biased permutation by
conditioning on $N$.


\subsection{Markov property in the limit} \label{subsec:parallel}

Results in \cite{pitmanyor} can be obtained as limits of those in
Section~\ref{subsec:joint.ppy}, including the Markov property and the
stick-breaking representation.
Consider a subordinator with L\'evy measure $\Lambda$, drift $\mathsf
{d} = 0$. Let $\tilde{T}_0$ be the subordinator at time 1. Assume
$\Lambda(1, \infty) < \infty$, $\Lambda(0, \infty) = \infty$,
$\int_0^1x\Lambda(\mathrm{d}x) < \infty$, and $\Lambda(\mathrm{d}x) = \rho(x) \,\mathrm{d}x$ for some
density $\rho$. Note that $\tilde{T}_0 < \infty$ a.s., and it has a
density determined by $\rho$ via its Laplace transform, which we denote
$\nu$. Let $\tilde{T}_k$ denote the remaining sum after removing the
first $k$ terms of the size-biased permutation of the sequence
$X^\downarrow$ of ranked jumps.

\begin{prop}[(\cite{pitmanyor})] \label{prop:sb.markovtransit}
The sequence $(\tilde{T}_0, \tilde{T}_1, \ldots)$ is a Markov chain
with stationary transition probabilities
\[
\P(\tilde{T}_{1} \in \mathrm{d}t_1| \tilde{T}_0 = t)
= \frac{t-t_1}{t}\cdot \rho (t-t_1) \frac{\nu(t_1)}{\nu(t)}
\,\mathrm{d}t_1.
\]
\end{prop}

Note the similarity to (\ref{eqn:T.tp}). Starting with (\ref
{eqn:joint.T}) and send $n \to\infty$, for any finite $k$, we have
$\nu
_{n-k} \to\nu$ pointwise, and by Theorem~\ref{thm:sb.classic},
$(n-k)\nu_1 \to\rho$ pointwise over $\R$, since there is no drift
term. Thus, the analogue of (\ref{eqn:joint.T}) in the limit is
\[
\P(\tilde{T}_0 \in \mathrm{d}t_0, \ldots, \tilde{T}_k
\in \mathrm{d}t_k) = \Biggl(\prod_{i=0}^{k-1}
\frac{t_i - t_{i+1}}{t_i}\rho(t_i - t_{i+1}) \Biggr)\nu
(t_k) \,\mathrm{d}t_0\cdots \mathrm{d}t_k.
\]
Rearranging gives the transition probability in Proposition~\ref
{prop:sb.markovtransit}.

Conditionally given $\tilde{T}_0 = t_0, \tilde{T}_1 = t_1, \ldots,
\tilde{T}_n = t_n$, the sequence of remaining terms in the size-biased
permutation $(X[n+1], X[n+2], \ldots)$ is distributed as
$(X^\downarrow
(1), X^\downarrow(2), \ldots)$ conditioned on $\sum_{i\geq
1}X^\downarrow(i) = t_n$, independent of the first $n$ size-biased
picks \cite{pitmanyor}, Theorem~4.2. The stick-breaking representation
in (\ref{eqn:stickbreak}) now takes the form
%
\begin{equation}
\label{eqn:stick2} \frac{X[k]}{\tilde{T}_0} = W_{k}\prod
_{i=1}^{k-1}\overline{W}_{i},
\end{equation}
where $X[k]$ is the $k$th size-biased pick, and $W_{i} = \frac
{X[i]}{\tilde{T}_{i-1}}$, $\overline{W}_{i} = 1 - W_{i} = \frac
{\tilde
{T}_i}{\tilde{T}_{i-1}}$. Proposition~\ref{prop:pt} and Corollary~\ref
{cor:pt} parallel the following result.

%
\begin{prop}[(McCloskey \cite{mc65} and Perman--Pitman--Yor \cite
{pitmanyor})]\label{prop:mccloskey}
The random variables $\tilde{T}_0$ and $W_{1}, W_2, \ldots$ in (\ref
{eqn:stick2}) are mutually independent if and only if $\tilde{T}_0$ has
distribution $\operatorname{gamma}(a, \lambda)$ for some $a, \lambda> 0$. In this
case, the $W_{i}$ are i.i.d. with distribution $\operatorname{beta}(1, a)$ for $i = 1,
2, \ldots.$
\end{prop}

\subsection{Invariance under size-biased permutation}\label{subsec:invariance}
We take a small detour to explain some results related to Propositions
\ref{prop:pt} and \ref{prop:mccloskey}
on characterization of size-biased permutations. For a random discrete
distribution prescribed by its probability mass function $P \in\Delta
_1$, let $P^\ast$ be its size-biased permutation. (Recall that $\Delta$
is the closed infinite simplex, $\Delta_1$ is its boundary. These are
defined at the beginning of Section~\ref{sec:sbp}.) Given $P \in
\Delta
_1$, one may ask when is there a $Q \in\Delta_1$ such that $Q =
P^\ast
$. Clearly $(P^\ast)^\ast= P^\ast$ for any $P \in\Delta_1$, thus this
question is equivalent to characterizing random discrete distributions
on $\mathbb{N}$ which are invariant under size-biased permutation (ISBP).
One such characterization is the following \cite{pitman96}, Theorem~4:
suppose $P \in\Delta_1$, $P_1 > 0$ a.s. Then $P = P^\ast$ if and only
if for each $k = 2, 3, \ldots,$ the function of $k$-tuples of positive integers
%
\begin{equation}
(n_1, \ldots, n_k) \mapsto\mathbb{E}\Biggl(\prod
_{i=1}^kP_i^{n_i-1}\prod
_{i=1}^{k-1}\Biggl(1 - \sum
_{j=1}^i P_j\Biggr)\Biggr)
\label{eqn:eppf}
\end{equation}
is a symmetric function of $n_1, \ldots, n_k$. Here is the
interpretation of this function using Kingman's paintbox. Sample $n =
n_1 + \cdots+ n_k$ points $\xi_1, \ldots, \xi_n$ one by one from the
uniform distribution on $[0,1]$. Each time a sample point discovers a
new component, write down its size. Conditioned on the event $P = p =
(p_1, p_2, \ldots)$, then
\[
\mathbb{E}\Biggl(\prod_{i=1}^kP_i^{n_i-1}
\prod_{i=1}^{k-1}\Biggl(1 - \sum
_{j=1}^i P_j\Biggr)\bigg | P = p\Biggr)
\]
is the probability that we discovered $k$ distinct paint boxes, where
the first box $p_1$ is rediscovered $n_1-1$ times, then we discover the
box $p_2$, which is then rediscovered $n_2-1$ times, and so on. Thus if
$P = P^\ast$, then the value of the function in (\ref{eqn:eppf}) is the
probability of the event $\xi_1 = \cdots= \xi_{n_1}$, $\xi_{n_1+1} =
\cdots= \xi_{n_1+n_2}$ and so on up to $\xi_{n-n_k+1} = \cdots= \xi
_n$, and that $\xi_1, \xi_{n_1+1}, \ldots, \xi_n$ are distinct.

Consider the stick-breaking representation of size-biased permutation,
that is, $P_n = W_1\cdots   W_{n-1} \overline{W}_n$. Suppose we want to
find ISBP distributions $P$ such that the $W_i$'s are independent. By
the above characterization, this is equivalent to finding such $P$ where
\[
\mathbb{E}\bigl(W_1^r\bar{W}_1^{s+1}
\bigr)\mathbb{E}\bigl(W_2^s\bigr) = \mathbb {E}
\bigl(W_1^s\bar {W}_1^{r+1}\bigr)
\mathbb{E}\bigl(W_2^r\bigr)
\]
for all pairs of non-negative integers $r$ and $s$. By analyzing this
equation, Pitman \cite{pitman96} proved a complete characterization of
ISBP in this case.

\begin{thmm}[(\cite{pitman96})]
Let $P \in\Delta_1$, $P_1 < 1$, and $P_n = W_1\cdots W_{n-1}
\overline
{W}_n$ for independent $W_i$. Then $P = P^\ast$ if and only if one of
the four following conditions holds.
\begin{enumerate}[1.]
\item[1.]$P_n \geq0$ a.s. for all $n$, in which case the distribution of
$W_n$ is
\[
\operatorname{beta}(1 - \alpha, \theta+ n\alpha)
\]
for every $n = 1, 2, \ldots,$ for some $0 \leq\alpha< 1$, $\theta>
-\alpha$.
\item[2.] For some integer constant $m$, $P_n \geq0$ a.s. for all $1 \leq
n \leq m$, and $P_n = 0$ a.s. otherwise. Then either
\begin{enumerate}[(a)]
\item[(a)] For some $\alpha> 0$, $W_n$ has distribution $\operatorname{beta}(1+\alpha,
m\alpha- n\alpha)$ for $n = 1, \ldots, m$;
or
\item[(b)]$W_n = 1/(m-n+1)$ a.s., that is, $P_n = 1/m$ a.s. for $n = 1,
\ldots, m$;
or
\item[(c)]$m = 2$, and the distribution $F$ on $(0,1)$ defined by $F(\mathrm{d}w) =
\bar{w}\mathbb{P}(W_1 \in \mathrm{d}w)/\E(\bar{W}_1)$ is symmetric about $1/2$.
\end{enumerate}
\end{enumerate}
\end{thmm}

The McCloskey case of Proposition~\ref{prop:mccloskey} is case 1 with
$\alpha= 0, \theta> 0$, and Patil--Taillie case of Proposition~\ref
{prop:pt} is case 2(a). These two cases are often written in the form
$W_{i}$ has distribution $\operatorname{beta}(1-\alpha, \theta+ i\alpha)$, $i = 1, 2,
\ldots$ for pairs of real numbers $(\alpha, \theta)$ satisfying either
$(0 \leq\alpha< 1, \theta> -\alpha)$ (case 1), or $(\alpha< 0,
\theta= m\alpha)$ for some $m = 1, 2, \ldots$ (case 2(a)). In both
settings, such a distribution $P$ is known as the $\operatorname{GEM}(\alpha, \theta)$
distribution. The abbreviation GEM was introduced by Ewens, which
stands for Griffiths--Engen--McCloskey. If $P$ is $\operatorname{GEM}(\alpha, \theta)$,
then $P^\downarrow$ is called a Poisson--Dirichlet distribution with
parameters $(\alpha, \theta)$, denoted $\operatorname{PD}(\alpha,\theta)$ \cite
{pitmanyor}.

In the McCloskey case of Proposition~\ref{prop:mccloskey}, the function
(\ref{eqn:eppf}) is the Donnelly--Tavare--Griffiths formula. If one
changes variables from $n_i$'s to $s_j$, where $s_j$ is the number of
$n_i$'s equal to $j$, then (\ref{eqn:eppf}) is the Ewens' sampling
formula \cite{ew72}. In studying this formula, Kingman \cite{ki78a}
initiated the theory of partition structures; see \cite{ghp} for recent
developments. Subsequent authors have studied partition structures and
their representations in terms of exchangeable random partitions,
random discrete distributions, random trees and associated random
processes of fragmentation and coalescence, Bayesian statistics, and
machine learning. See \cite{csp} and references therein.


\section{Asymptotics of the last $u$ fraction of the size-biased
permutation}\label{sec:asymp}
In this section, we derive Glivenko--Cantelli and Donsker-type theorems
for the distribution of the last $u$ fraction of terms in a finite
i.i.d. size-biased permutation. It is especially convenient to work
with the induced order statistics representation since we can appeal to
tools from empirical process theory. In particular, our results are
special cases of more general statements which hold for arbitrary
induced order statistics in $d$ dimensions (see Section~\ref{sec:history}). Features pertaining to i.i.d. size-biased permutation
are presented in Lemma~\ref{lem:ode}. The proof is a direct computation.
We first discuss the interesting successive sampling interpretation of
Lemma~\ref{lem:ode}, quoting some results needed to make the discussion
rigorous. We then derive the aforementioned theorems and conclude with
a brief historical account of induced order statistics.

%
\begin{lem}\label{lem:ode} Suppose $F$ has support on $(0, \infty)$,
finite mean. For $u \in(0,1)$, define
%
\begin{equation}
\label{eqn:Fu} F_u(\mathrm{d}x) = \frac{\mathrm{e}^{-x\phi^{-1}(u)}}{u} F(\mathrm{d}x)
\end{equation}
and extend the definition to $\{0, 1\}$ by continuity, where $\phi$ is
the Laplace transform of $F$ as in Proposition~\ref{lem:joint.x.u}.
Then $F_u$ is a probability distribution on $(0, \infty)$ for all $u
\in[0,1]$, and $G_u$ in (\ref{eqn:Gu}) satisfies
%
\begin{equation}
\label{eqn:Gu.is.sb} G_u(\mathrm{d}x) = xF_u(\mathrm{d}x)/ \mu_u,
\end{equation}
where $\mu_u = \int x F_u(\mathrm{d}x) = \frac{-\phi'(\phi^{-1}(u))}{u}$.
Furthermore,
%
\begin{equation}
\label{eqn:integral} \int_0^u G_s(\mathrm{d}x) \,\mathrm{d}s
= F_u(\mathrm{d}x)
\end{equation}
for all $s \in[0,1]$. In other words, the density
\[
f(u,x) = F_u(\mathrm{d}x) / F(\mathrm{d}x) = u^{-1}\mathrm{e}^{-x\phi^{-1}(u)}
\]
of $F_u$ with respect to $F$ solves the differential equation
%
\begin{equation}
\label{eqn:evo} \frac{\mathrm{d}}{\mathrm{d}u}\bigl[uf(u,x)\bigr] = \frac{-xf(u,x)}{\mu_u}
\end{equation}
with boundary condition $f(1,x) \equiv1$.
\end{lem}

For any distribution $F$ with finite mean $\mu$ and positive support,
$xF(\mathrm{d}x)/\mu$ defines its \emph{size-biased distribution}. If $F$ is the
empirical distribution of $n$ positive values $x_n(1), \ldots, x_n(n)$,
for example, one can check that $xF(\mathrm{d}x)/\mu$ is precisely the
distribution of the first size-biased pick $X_n[1]$. For continuous
$F$, the name size-biased distribution is justified by the following lemma.

\begin{lem}\label{lem:size.biased.pick} Consider an i.i.d. size-biased
permutation $(X_n[1], \ldots, X_n[n])$ from a distribution $F$ with
support on $[0, \infty)$ and finite mean $\mu$. Then
\[
\lim_{n\to\infty}\mathbb{P}\bigl(X_n[1] \in \mathrm{d}x\bigr) =
\frac{xF(\mathrm{d}x)}{\mu}.
\]
\end{lem}

Since the $\lfloor nu \rfloor$th smallest out of $n$ uniform order
statistics converge to $u$ as $n \to\infty$, $G_u$ is the limiting
distribution of $X_n^{\mathrm{rev}}[\lfloor nu \rfloor]$, the size-biased pick
performed when a $u$-fraction of the sequence is left. By (\ref
{eqn:Gu.is.sb}), $G_u$ is the size-biased distribution of $F_u$. Thus,
$F_u$ can be interpreted as the limiting distribution of the remaining
$u$-fraction of terms in a successive sampling scheme. This intuition
is made rigorous by Corollary~\ref{cor:emp} below. In words, it states
that $F_u$ is the limit of the empirical distribution function (e.d.f.)
of the last $u > 0$ fraction in a finite i.i.d. size-biased permutation.

\begin{cor}\label{cor:emp}For $u \in(0,1]$, let $F_{n,u}(\cdot)$
denote the empirical distribution of the last $\lfloor nu \rfloor$
values of an i.i.d. size-biased permutation with length $n$. For each
$\delta\in(0,1)$, as $n \to\infty$,
%
\begin{equation}
\label{emp.u} \sup_{u \in[\delta, 1]} \sup_\I\bigl|
F_{n,u}(\I) - F_u(\I)\bigr| \stackrel {a.s.} {\longrightarrow}
0,
\end{equation}
where $\I$ ranges over all subintervals of $(0, \infty)$.
\end{cor}

Therefore in the limit, after removing the first $1-u$ fraction of
terms in the size-biased permutation, we are left with an (infinitely)
large sequence of numbers distributed like i.i.d. draws from $F_u$,
from which we do a size-biased pick, which has distribution $G_u(\mathrm{d}x) =
xF_u(\mathrm{d}x)/\mu_u$ as specified by (\ref{eqn:Gu.is.sb}).

Since $X_n^{\mathrm{rev}}[\lfloor nu \rfloor]$ converges in distribution to
$G_u$ for $u \in[0, 1]$, Corollary~\ref{cor:emp} lends a sampling
interpretation to Lemma~\ref{lem:ode}. Equation (\ref{eqn:evo}) has the
heuristic interpretation as characterizing the evolution of the mass at
$x$ over time $u$ in a successive sampling scheme. To be specific,
consider a successive sampling scheme on a large population of $N$
individuals, with species size distribution~$H$. Scale time such that
at time $u$, for $0 \leq u \leq1$, there are $Nu$ individuals (from
various species) remaining to be sampled. Let $H_u$ denote the
distribution of species sizes at time $u$, and fix the bin $(x, x +
\mathrm{d}x)$ of width $\mathrm{d}x$ on $(0, \infty)$. Then $NuH_u(\mathrm{d}x)$ is the number of
individuals whose species size lie in the range $(x, x + \mathrm{d}x)$ at time
$u$. Thus, $\frac{\mathrm{d}}{\mathrm{d}u}NuH_u(\mathrm{d}x)$ is the rate of individuals to be
sampled from this range of species size at time $u$. The probability\vspace*{1pt} of
an individual whose species size is in $(x, x+\mathrm{d}x)$ being sampled at
time $u$ is $\frac{xH_u(\mathrm{d}x)}{\int_0^\infty xH_u(\mathrm{d}x)}$. As we scaled
time such that $u \in[0,1]$, in time $\mathrm{d}u$ we sample $N \,\mathrm{d}u$
individuals. Thus,
\[
\frac{\mathrm{d}}{\mathrm{d}u}NuH_u(\mathrm{d}x) = -N \frac{xH_u(\mathrm{d}x)}{\int_0^\infty xH_u(\mathrm{d}x) }.
\]
Let $f(u,x) = H_u(\mathrm{d}x) / H(\mathrm{d}x)$, then as a function in $u$, the above
equation reduces to (\ref{eqn:evo}).

\begin{ex} Suppose $F$ puts probability $p$ at $a$ and $1-p$ at $b$,
with $a < b$. Let $p(u) = F_u(a)$ be the limiting fraction of $a$ left
when proportion $u$ of the sample is left. Then the evolution equation~(\ref{eqn:evo}) becomes
\begin{eqnarray*}
p'(u) &=& u^{-1} \biggl( \frac{a}{ap(u) + b(1-p(u))} - 1 \biggr)
p(u)
\\
&=& u^{-1} \biggl( \frac{a - a p(u) - b(1-p(u))}{b - (b-a)p(u)} \biggr) p(u)
\\
&=& u^{-1} \frac{(1-p(u))p(u)(a-b)}{b - (b-a)p(u)}, \qquad 0 \leq u \leq1,
\end{eqnarray*}
with boundary condition $p(0) = p$. To solve for $p(u)$, let $y$ solve
$u = p\mathrm{e}^{-ay} + (1-p) \mathrm{e}^{-by}$. Then $p(u) = p\mathrm{e}^{-ay}/u$.
\end{ex}

\subsection{A Glivenko--Cantelli theorem}
We now state a Glivenko--Cantelli-type theorem which applies to
size-biased permutations of finite deterministic sequences. Versions of
this result are known in the literature \cite{bnw,holst,gordon83,rosen}, see discussions in Section~\ref{sec:history}. We offer an
alternative proof using induced order statistics.

\begin{thmm}\label{thm:gc} Let $(x_n, n = 1, 2, \ldots)$ be a
deterministic triangular array of positive numbers with corresponding
c.d.f. sequence $(E_n, 1 \leq n)$. Suppose
%
\begin{equation}
\sup_x\bigl|E_n(x) - F(x)\bigr| \to0 \qquad\mbox{as } n \to
\infty\label
{eqn:sup.condition}
\end{equation}
for some distribution $F$ on $(0, \infty)$. Let $u \in(0, 1]$. Let
$E_{n,u}(\cdot)$ be the empirical distribution of the last $\lfloor nu
\rfloor$ terms in a size-biased permutation of the sequence $x_n$. Then
for each $\delta\in(0,1)$,
%
\begin{equation}
\label{eqn:gc.main} \sup_{u \in[\delta, 1]} \sup_\I\bigl|
E_{n,u}(\I) - F_u(\I)\bigr| \stackrel {a.s.} {\longrightarrow} 0\qquad
\mbox{as } n \to\infty,
\end{equation}
where $\I$ ranges over all subintervals of $(0, \infty)$.
\end{thmm}

We state the theorem in terms of convergence in distribution of the
last $u$ fraction of terms in a successive sampling scheme. Since $E_n
\to F$ uniformly, an analogous result holds for the distribution of the
first $(1-u)$ fraction. The last $u$ fraction is more interesting due
to the heuristic interpretation of $F_u$ in the previous section.

\begin{pf*}{Proof of Theorem \ref{thm:gc}} Define $Y_n(i) = \varepsilon_n(i)/x_n(i)$ for $i = 1, \ldots
, n$ where $\varepsilon_n(i)$ are i.i.d. standard exponentials as in
Proposition~\ref{prop:coupling}. Let $H_n$ be the empirical
distribution function (e.d.f.) of the $Y_n(i)$,
\[
H_n(y) := \frac{1}{n}\sum_{i=1}^n
\mathbf{1}_{\{Y_n(i) < y\}}.
\]
Let $J_n$ denote the e.d.f. of $(x_n(i), Y_n(i))$. By Proposition~\ref
{lem:joint.x.u},
%
\begin{eqnarray}
\label{eqn:eemp} E_{n,u}(\mathbf{I}) &=& \frac{n}{\lfloor nu \rfloor}
\frac{1}{n} \sum_{i=1}^n
\mathbf{1}_{\{x_n(i) \in\mathbf{I}\}}\mathbf{1}_{\{Y_n(i) >
H_n^{-1}(1-u)\}}
\nonumber
\\[-8pt]
\\[-8pt]
\nonumber
&=& \frac{n}{\lfloor nu \rfloor}J_n
\bigl(\mathbf{I} \times \bigl[H_n^{-1}(1-u),
H_n^{-1}(1)\bigr]\bigr).
\end{eqnarray}
Fix $\delta\in(0, 1)$, and let $u \in[\delta, 1]$. Let $\phi$ be the
Laplace transform of $F$ and $J$ the joint law of $(X, \varepsilon/X)$,
where $X$ is a random variable with distribution $F$, and $\varepsilon$ is
an independent standard exponential. Note that $\frac{1}{u}J(\mathbf{I}
\times[\phi^{-1}(u), \infty)) = F_u(\mathbf{I})$.
Thus,
%
\begin{eqnarray}\label{eqn:split}
\E_{n, u}(\mathbf{I}) - F_u(\mathbf{I})& = & \biggl(
\frac{n}{\lfloor nu
\rfloor}J_n\bigl(\mathbf{I} \times\bigl[H_n^{-1}(1-u),
H_n^{-1}(1)\bigr]\bigr) - \frac
{n}{\lfloor nu \rfloor}J_n
\bigl(\mathbf{I} \times\bigl[\phi^{-1}(u), \infty \bigr)\bigr)\biggr )
\nonumber
\\[-8pt]
\\[-8pt]
\nonumber
& &{}+ \biggl(\frac{n}{\lfloor nu \rfloor}J_n\bigl(\mathbf{I} \times\bigl[\phi
^{-1}(u), \infty\bigr)\bigr) - \frac{1}{u}J\bigl(\mathbf{I}
\times\bigl[\phi^{-1}(u), \infty\bigr)\bigr) \biggr).
\end{eqnarray}
Let us consider the second term. Note that
\[
J_n\bigl(\mathbf{I} \times\bigl[\phi^{-1}(u), \infty\bigr)\bigr) =
\int_0^\infty \mathrm{e}^{-t\phi
^{-1}(u)}\mathbf{1}_{\{t \in\mathbf{I}\}}
E_n(\mathrm{d}t).
\]
Since $E_n$ converges to $F$ uniformly and $\mathrm{e}^{-t\phi^{-1}(u)}$ is
bounded for all $t \in(0, \infty)$ and $u \in[\delta, 1]$,
\[
\sup_{u \in[\delta, 1]}\sup_\mathbf{I}\biggl\llvert
\frac{n}{\lfloor nu
\rfloor
}J_n\bigl(\mathbf{I} \times\bigl[\phi^{-1}(u),
\infty\bigr)\bigr) - \frac
{1}{u}J\bigl(\mathbf {I} \times\bigl[\phi^{-1}(u),
\infty\bigr)\bigr)\biggr\rrvert \stackrel {\mathrm{a.s.}} {\longrightarrow} 0
\qquad\mbox{as } n \to
\infty.
\]
Let us consider the first term. Recall that $H_n^{-1}(1) = \max_{i=1,\ldots,n} Y_i$, thus $H_n^{-1}(1) \to\infty$ as $n \to\infty
$ a.s.
Since $J_n$ is continuous in the second variable, it is sufficient to
show that
%
\begin{equation}
\label{eqn:hncont} \sup_{u \in[\delta, 1]}\bigl|H_n^{-1}(1-u)
- \phi^{-1}(u)\bigr| \stackrel {\mathrm{a.s.}} {\longrightarrow} 0\qquad \mbox{as } n \to
\infty.
\end{equation}
To achieve this, let $A_n$ denote the `average' measure
\[
A_n(y) := \frac{1}{n}\sum_{i=1}^n
\P\bigl(Y_n(i) < y\bigr) = 1 - \int_0^\infty
\mathrm{e}^{-xy} \,\mathrm{d}E_n(x).
\]
A theorem of Wellner \cite{wellner}, Theorem~1, states that if the
sequence of measures $(A_n, n \geq1)$ is tight, then the Prohorov
distance between $H_n$ and $A_n$ converges a.s. to $0$ and $n \to
\infty
$. In this case, since $E_n$ converges to $F$ uniformly, $A_n$
converges uniformly to $1 - \phi$. Thus $H_n$ converges uniformly to
$1-\phi$, and (\ref{eqn:hncont}) follows.
\end{pf*}
%

\begin{pf*}{Proof of Corollary~\ref{cor:emp}}
When $E_n$ is the e.d.f. of $n$ i.i.d. picks from $F$, then (\ref
{eqn:sup.condition}) is satisfied a.s. by the Glivenko--Cantelli
theorem. Thus, Theorem~\ref{thm:gc} implies Corollary~\ref{cor:emp}.
\end{pf*}

\begin{pf*}{Proof of Lemma~\ref{lem:size.biased.pick}}
Let $\phi$ be
the Laplace transform of $F$. For $y > 0$,
\[
\frac{\mathrm{d}\phi(y)^n}{\mathrm{d}y} = n\phi(y)^{n-1} \phi'(y).
\]
By Corollary~\ref{cor:marginal}, we have
\[
\frac{\mathbb{P}(X_n[1] \in \mathrm{d}x)}{xF(\mathrm{d}x)} = n\int_0^\infty
\mathrm{e}^{-xy}\phi (y)^{n-1} \,\mathrm{d}y = \int_0^\infty
\frac{\mathrm{e}^{-xy}}{\phi'(y)}\frac
{\mathrm{d}}{\mathrm{d}y}\bigl(\phi (y)^n\bigr) \,\mathrm{d}y.
\]
Apply integration by parts, the constant term is
\[
\frac{\mathrm{e}^{-xy}}{\phi'(y)}\phi(y)^n\bigg|_0^\infty= -
\frac{1}{\phi
'(0)} = \frac{1}{\mu}.
\]
The integral term is
\[
\int_0^\infty\frac{\mathrm{d}}{\mathrm{d}y}
\bigl(\mathrm{e}^{-xy}\phi'(y)\bigr) \bigl(\phi(y)
\bigr)^n \,\mathrm{d}y.
\]
The integrand is integrable for all $n$, thus
\[
\lim_{n\to\infty}\int_0^\infty
\frac{\mathrm{d}}{\mathrm{d}y}\bigl(\mathrm{e}^{-xy}\phi'(y)\bigr) \bigl(\phi
(y)\bigr)^n \,\mathrm{d}y =\int_0^\infty\lim
_{n\to\infty}\frac{\mathrm{d}}{\mathrm{d}y}\bigl(\mathrm{e}^{-xy}\phi
'(y)\bigr) \bigl(\phi(y)\bigr)^n \,\mathrm{d}y = 0.
\]
Since $n \to\infty$, $\phi(y)^n \to0$ for all $y > 0$. Therefore,
$
\lim_{n\to\infty}\frac{\mathbb{P}(X_n[1] \in \mathrm{d}x)}{xF(\mathrm{d}x)} =
\frac
{1}{\mu}$.
\end{pf*}

\subsection{Historical notes on induced order statistics and successive
sampling} \label{sec:history}
Induced order statistics were first introduced by David \cite{david73}
and independently by Bhattacharya~\cite{bhat74}. Typical applications
stem from modeling an indirect ranking procedure, where subjects are
ranked based on their $Y$-attributes although the real interest lies in
ranking their $X$-attributes, which are difficult to obtain at the
moment where the ranking is required.\footnote{One often uses $X$ for
the variable to be ordered, and $Y$ for the induced variable, with the
idea that $Y$ is to be predicted. Here we use $X$ for the induced order
statistics since $X_n[k]$ has been used for the size-biased
permutation. The role of $X$ and $Y$ in our case is interchangeable, as
evident when one writes $X_n(i)Y_n(i) = \varepsilon_n(i)$.} For example,
in cattle selection, $Y$ may represent the genetic makeup, for which
the cattle are selected for breeding, and $X$ represents the milk
yields of their female offspring. Thus a portion of this literature
focuses on comparing distribution of induced order statistics to that
of usual order statistics \cite{dayang77,yang77,nagada94,hall94}.
The most general statement on asymptotic distributions is obtained by
Davydov and Egorov \cite{egorov89}, who proved the functional central
limit theorem and the functional law of the iterated logarithm for the
process $S_{n,u}$ under tight assumptions. The functional central limit
theorem for i.i.d. size-biased permutation is a special case of their
results. Various versions of results in Section~\ref{sec:asymp},
including the functional central limit theorem, are also known in the
successive sampling community \cite{holst,rosen,gordon83,bnw,sen}.
For example, Bickel, Nair and Wang \cite{bnw} proved Theorem~\ref
{thm:gc} with convergence in probability when $E_n$ and $F$ have the
same discrete support on finitely many values.

\section{Poisson coupling of size-biased permutation and order
statistics}\label{sec:ppp}

Comparisons between the distribution of induced order statistics and
order statistics of the same sequence have been studied in the
literature \cite{dayang77,yang77,nagada94,hall94}. However, finite
i.i.d. size-biased permutation has the special feature that there
exists an explicit coupling between these two sequences as described in
Proposition~\ref{prop:coupling}. Using this fact, we now derive
Theorem~\ref{thm}, which gives a Poisson coupling between the \emph
{last} $k$ size-biased terms $X_n^{\mathrm{rev}}[1], \ldots, X_n^{\mathrm{rev}}[k]$ and
the $k$ smallest order statistics $X_n^{\uparrow}(1), \ldots,
X^{\uparrow
}_n(k)$ as $n \to\infty$.
The existence of a Poisson coupling is not surprising, since the
increasing sequence of order statistics $(X_n^{\uparrow}(1),
X_n^{\uparrow}(2), \ldots)$ converges to points in a Poisson point
process (p.p.p.) whose intensity measure depends on the behavior of $F$
near the infimum of its support, which is $0$ in our case. This
standard result in order statistics and extreme value theory dates back
to R\'{e}nyi \cite{renyi}, and can be found in \cite{evtHaan}.

\subsection{Random permutations from Poisson scatter}
Let $N(\cdot)$ be a Poisson scatter on $(0, \infty)^2$. Suppose
$N(\cdot
)$ has intensity measure $m$ such that for all $s, t \in(0, \infty)$
\[
m\bigl((0, s) \times(0, \infty)\bigr) < \infty,\qquad m\bigl((0, \infty) \times(0, t)
\bigr) < \infty.
\]
Then one obtains a random permutation of $\mathbb{N}$ from ranking
points $(x(i), y(i))$ in $N$ according to either the $x$ or $y$
coordinate. Let $x^\ast$ and $y^\ast$ denote the induced order
statistics of the sequence $x$ and $y$ obtained by ranking points by
their $y$ and $x$ values in increasing order, respectively. For $j, k =
1, 2, \ldots,$ define sequences of integers $(K_j), (J_k)$ such that
$x^{\uparrow}(J_k) = x^{\ast}(k)$, $y^{\uparrow}(K_j) = y^{\ast}(j)$;
see Figure~\ref{fig:scatter}.

\begin{figure}

\includegraphics{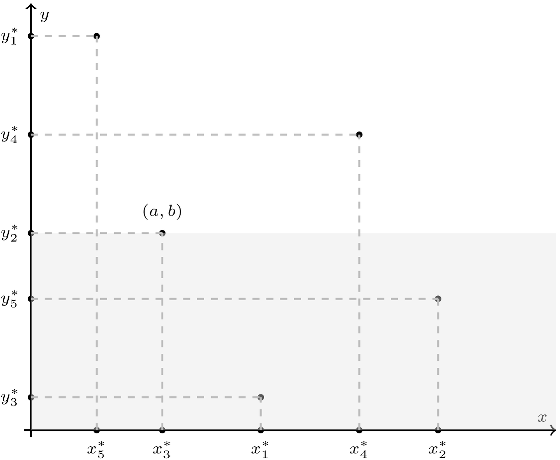}

\caption{A point scatter on the plane. Here $J_1 = 5, J_2 = 3, J_3 = 1,
J_4 = 4, J_5 = 2$, and $K_1 = 3,  K_2 = 5, K_3 = 2, K_4 = 4, K_5 = 1$.
The permutations $J$ and $K$ are inverses. Conditioned on $x^\uparrow
(2) = a$ and $y^\ast(2) = b$, the number of points lying in the shaded
region determines $J_2 - 1$.}
\label{fig:scatter}
\end{figure}

For $j \geq1$, conditioned on $x(j) = s, y^{\ast}(j) = t$,
%
\begin{equation}
\label{eqn:kj} K_j - 1 \stackrel{d} {=} \operatorname{Poisson} \bigl(m\bigl((s,
\infty) \times(0, t)\bigr) \bigr) + \operatorname{Binomial} \biggl(j-1, \frac{m((0, s) \times(0, t))}{m((0, s) \times
(0, \infty))}
\biggr),
\end{equation}
where the two random variables involved are independent. Similarly, for
$k \geq1$, conditioned on $x^\ast_k = s$, $y(k ) = t$,
%
\begin{equation}
\label{eqn:jk} J_k - 1 \stackrel{d} {=} \operatorname{Poisson} \bigl(m\bigl((0,s)
\times(t,\infty )\bigr) \bigr) + \operatorname{Binomial} \biggl(k-1, \frac{m((0, s) \times(0, t))}{m((0, \infty)
\times(t, \infty))} \biggr),
\end{equation}
where the two random variables involved are independent. When $m$ is a
product measure, it is possible to compute the marginal distribution of
$K_j$ and $J_k$ explicitly for given $j, k \geq1$. 

Random permutations from Poisson scatters appeared in \cite{pitmanyor}, Section~4.
When $X^\downarrow$ is the sequence of ranked jumps of a
subordinator, these authors noted that one can couple the size-biased
permutation with the order statistics via the following p.p.p.
%
\begin{equation}
\label{eqn:ppois} N(\cdot) := \sum_{k\geq1}\mathbf{1}
\bigl[\bigl(X[k],Y^{\uparrow}(k)\bigr) \in \cdot\bigr] = \sum
_{k\geq1}\mathbf{1}\bigl[\bigl(X^{\downarrow}(k),Y(k)\bigr) \in
\cdot\bigr],
\end{equation}
where $Y$ is an independent sequence of standard exponentials. Thus,
$N(\cdot)$ has intensity measure $ m(\mathrm{d}x\,\mathrm{d}y) = x\mathrm{e}^{-xy} \Lambda(\mathrm{d}x) \,\mathrm{d}y$. The first expression in (\ref{eqn:ppois}) defines a scatter of
$(x,y)$ values in the plane listed in increasing $y$ values, and the
second represents the same scatter listed in decreasing $x$ values.
Since $\sum_{i\geq1}X^\downarrow(i) < \infty$ a.s., the $x$-marginal
of the points in (\ref{eqn:ppois}) has the distribution of the
size-biased permutation $X^\ast$, since it prescribes the joint
distribution of the first $k$ terms $X[1], \ldots, X[k]$ of $X^\ast$
for any finite $k$.
Perman, Pitman and Yor used this p.p.p. representation to generalize
size-biased permutation to $h$-biased permutation, where the `size' of
a point $x$ is replaced by an arbitrary strictly positive function
$h(x)$; see \cite{pitmanyor}, Section~4.

\subsection{Poisson coupling in the limit}
Our theorem states that in the limit, finite i.i.d. size-biased
permutation is a form of random permutation obtained from a Poisson
scatter with a certain measure, which, under a change of coordinate, is
given by (\ref{eqn:measure.mu}). Before stating the theorem, we need
some technical results. The distribution of the last few size-biased
picks depends on the behavior of $F$ near $0$, the infimum of its
support. We shall consider the case where $F$ has `power law' near $0$,
like that of a Gamma distribution.

%
\begin{lem} \label{lem:approx.G.by.Gamma}
Suppose $F$ is supported on $(0, \infty)$ with Laplace transform $\phi
$. Let $u = \phi(y)$, $X_u$ a random variable distributed as $G_u(\mathrm{d}x)$
defined in (\ref{eqn:Gu}). For $\lambda, a > 0$,
%
\begin{equation}
F(x) \sim\frac{\lambda^a x^a}{\Gamma(a + 1)} \qquad\mbox{as } x \to0, \label{eqn:assump.F}
\end{equation}
if and only if,
%
\begin{equation}
\phi(y)  \sim\lambda^a/y^{a} \qquad\mbox{as } y \to \infty
\label{eqn:phi.asymp}.
\end{equation}
Furthermore, (\ref{eqn:assump.F}) implies
%
\begin{equation}
u^{-1/a}X_u  \stackrel{d} {\longrightarrow} \operatorname{gamma}(a+1,
\lambda) \qquad\mbox{as } u \to0. \label{eqn:xu.gamma}
\end{equation}
\end{lem}

\begin{pf} The equivalence of (\ref{eqn:assump.F}) and (\ref
{eqn:phi.asymp}) follows from a version of Karamata Tauberian
theorem~\cite{regvar}, Section~1.7. Assume (\ref{eqn:assump.F}) and (\ref
{eqn:phi.asymp}). We shall prove (\ref{eqn:xu.gamma}) by looking at the
Laplace transform of the non-size-biased version $X_u'$, which has
distribution $F_u$. For $\theta\geq0$,
%
\begin{equation}
\label{eqn:laplace.xpu} E\bigl(\exp\bigl(-\theta X_u'\bigr)
\bigr) = \int_0^\infty u^{-1}\exp(-yx -
\theta x) F(\mathrm{d}x) = u^{-1}\phi(y + \theta) = \frac{\phi(y+\theta)}{\phi(y)}.
\end{equation}
Now as $y \to\infty$ and $u = \phi(y) \to0$, for each fixed $\eta>
0$, (\ref{eqn:phi.asymp}) implies
\[
E\bigl(\exp\bigl(-\eta u^{-1/a}X_u'\bigr)
\bigr) = \frac{\phi(y + \eta\phi
(y)^{-1/a})}{\phi(y)} \sim\frac{\lambda^a(y+\eta\lambda^{-1} y)^{-a} }{\lambda^ay^{-a}} = \biggl(\frac{\lambda}{\lambda+\eta}
\biggr)^a.
\]
That is to say
%
\begin{equation}
\label{eqn:xuprime.gamma} u^{-1/a} X_u' \stackrel{d} {
\longrightarrow} \operatorname{gamma}(a, \lambda).
\end{equation}
Since $\phi$ is differentiable, (\ref{eqn:laplace.xpu}) implies
$E(X'_u) = \phi'(y) / \phi(y)$. Now $\phi$ has an increasing derivative~$\phi'$, thus (\ref{eqn:phi.asymp}) implies $\phi'(y) \sim a\lambda
^a/y^{a+1}$ as $y \to\infty$. Therefore,
\[
u^{-1/a}E\bigl(X_u'\bigr) = \frac{\phi'(y)}{\phi(y)^{1 + 1/a}}
\to\frac
{a}{\lambda},
\]
which is the mean of a $\operatorname{gamma}(a,\lambda)$ random variable. Thus, the
random variables $u^{-1/a}X'_u$ are uniformly integrable, so for any
bounded continuous function $h$, we can compute
\[
E\bigl(h\bigl(u^{-1/a}X_u\bigr)\bigr) = \frac
{E[(u^{-1/a}X_u')h(u^{-1/a}X_u')]}{u^{-1/a}E(X_u')}
\to\frac{E[\gamma
_{a,\lambda}h(\gamma_{a,\lambda})]}{E(\gamma_{a,\lambda})} = E\bigl(h(\gamma _{a+1,\lambda})\bigr),
\]
where $\gamma_{b,\lambda}$ is a $\operatorname{gamma}(b,\lambda)$ random variable.
This proves (\ref{eqn:xu.gamma}).
\end{pf}

We now present the analogue of (\ref{eqn:fidi.sbs}) for the last few
size-biased picks $X_n^{\mathrm{rev}}[1], \ldots, X_n^{\mathrm{rev}}[k]$ and the promised
Poisson coupling.

%
\begin{thmm} \label{thm} Suppose that (\ref{eqn:assump.F}) holds for
some $\lambda, a > 0$. Let $N(\cdot)$ be a Poisson scatter on
$(0,\infty)^2$ with intensity measure
%
\begin{equation}
\label{eqn:measure.mu} \frac{\mu(\mathrm{d}s \,\mathrm{d}t)}{\mathrm{d}s \,\mathrm{d}t} = \frac{1}{a} \Gamma
(a+1)^{1/a}(s/t)^{1/a}\exp\bigl\{-\bigl(\Gamma(a+1)s/t
\bigr)^{1/a}\bigr\}.
\end{equation}
By ranking points in either increasing $T$ or $S$ coordinate, one can write
%
\begin{equation}
\label{16.b} N(\cdot) = \sum_k\1\bigl[
\bigl(S(k), T^\uparrow(k)\bigr) \in\cdot\bigr] = \sum
_j\1 \bigl[\bigl(S^\uparrow(j), T(j)\bigr) \in\cdot
\bigr].
\end{equation}
Define $\Psi_{a,\lambda}(s) = s^{1/a}\Gamma(a+1)^{1/a}/\lambda$. Define
a sequence of random variables $\xi$ via
%
\begin{equation}
\label{eqn:xik} \xi(k) = \Psi_{a,\lambda}\bigl(S^\uparrow(k)\bigr),
\end{equation}
and let $\xi^\ast$ be its reordering defined $\xi^\ast(k) = \Psi
_{a,\lambda}(S(k))$. Then jointly as $n \to\infty$,
%
\begin{eqnarray}
n^{1/a} X^\uparrow_n &\stackrel{{f.d.d.}} {
\longrightarrow}& \xi, \label
{eqn:fidi.order}
\\
n^{1/a} \bigl(X_n^\ast\bigr)^{\mathrm{rev}} &
\stackrel{{f.d.d.}} {\longrightarrow}& \xi ^\ast. \label{15.a}
\end{eqnarray}
In particular, for each $n$, let $J_n = (J_{nk}, 1 \leq k \leq n)$ be
the permutation of $\{1, \ldots, n\}$ defined by $ X^{\mathrm{rev}}_n[k] =
X_n(J_{nk})$. As $n \to\infty$,
%
\begin{equation}
\label{15.g} (J_{nk}, 1 \leq k \leq n ) \stackrel {{f.d.d.}} {
\longrightarrow} (J_k: 1 \leq k < \infty ),
\end{equation}
where $J_k$ is the random permutation of $\{1, 2, \ldots \}$, defined by
%
\begin{equation}
\label{15.h} \xi^\ast(k) = \xi(J_k)
\end{equation}
for $k = 1, 2, \ldots,$ and the f.d.d. convergence in (\ref
{eqn:fidi.order}), (\ref{15.a}), (\ref{15.g}) all hold jointly.
\end{thmm}

In other words, the Poisson point process $N(\cdot)$ defined in (\ref
{16.b}) with measure (\ref{eqn:measure.mu}) defines a random
permutation $(J_k)$ of $\mathbb{N}$ and its inverse $(K_j)$. Theorem~\ref{thm} states that $(J_k)$ is precisely the limit of the random
permutation induced by the size-biased permutation of a sequence of $n$
i.i.d. terms from $F$. Furthermore, to obtain the actual sequence of
size-biased permutation, one only needs to apply the \emph
{deterministic} transformation $\Psi_{a,\lambda}$ to the sequence of
$s$-marginals of points in $N(\cdot)$, ranked according to their
$t$-values. The sequence of increasing order statistics can be obtained
by applying the transformation $\Psi_{a,\lambda}$ to the $s$-marginals
ranked in increasing order.


%
\begin{pf*}{Proof of Theorem~\ref{thm}}
By Lemma~\ref
{lem:approx.G.by.Gamma}, it is sufficient to prove the theorem for the
case $F$ is $\operatorname{gamma}(a, \lambda)$. First, we check that the sequence on
the right-hand side of (\ref{eqn:fidi.order}) and (\ref{15.a}) have the
right distribution. Indeed, by standard results in order statistics
\cite{evtHaan}, Theorem~2.1.1, as $n \to\infty$, the sequence
$n^{1/a}X_n^\uparrow$ converges (f.d.d.) to the sequence $\widetilde{\xi
}$, where
%
\begin{equation}
\label{eqn:fidi.sbs} \widetilde{\xi}(k) = \bigl(S^\uparrow(k)\bigr)^{1/a}
\Gamma(a+1)^{1/a}/\lambda= \Psi _{a,\lambda}\bigl(S^\uparrow(k)
\bigr),
\end{equation}
where $S^\uparrow(k) = \varepsilon_1 + \cdots+ \varepsilon_k$ for
$\varepsilon
_i$ i.i.d. standard exponentials. Similarly, by Proposition~\ref
{prop:newgb} and law of large numbers, the sequence $n^{1/a}(X_n^\ast
)^{\mathrm{rev}}$ converges (f.d.d.) to the sequence $\widetilde{\xi}^\ast$, where
\[
\widetilde{\xi}^\ast(k) = \bigl(T^\uparrow(k)
\bigr)^{1/a}\gamma_k/\lambda,
\]
where $T^\uparrow(k) = \varepsilon_1' + \cdots+ \varepsilon_k'$ for i.i.d.
standard exponentials $\varepsilon_i'$, and $\gamma_k$, $k = 1, \ldots, n$
are i.i.d. $\operatorname{gamma}(a+1,1)$, independent of the $T(k)$. By direct
computation, we see that $\widetilde{\xi} \stackrel{d}{=} \xi$ and
$\widetilde{\xi}^\ast \stackrel{d}{=} \xi^\ast$. The dependence between
the two sequences $S$ and $T$ comes from Proposition~\ref
{prop:coupling}, which tells us that $S(k)$ is the term $S^\uparrow(j)$
that is paired with $T^\uparrow(k)$ in our Poisson coupling. Observe
$\Psi_{a,\lambda}$ has inverse function $\Psi_{a,\lambda}^{-1}(x) =
\lambda^a x^a / \Gamma(a+1)$. Thus applying (\ref{eqn:fidi.sbs}), we have
%
\begin{equation}
\label{16.a} S(k) = \Psi_{a,\lambda}^{-1}\bigl(\widetilde{
\xi}^\ast(k)\bigr) = \lambda^a \bigl[\widetilde{
\xi}^\ast(k)\bigr]^a/\Gamma(a+1) = T{(k)}
\gamma_k^a / \Gamma(a+1).
\end{equation}
Comparing (\ref{eqn:fidi.sbs}) and (\ref{16.a}) gives a pairing between
$S(k)$ and $S^\uparrow(k)$, and hence $T^\uparrow(k)$ and $S^\uparrow
(k)$, via $\widetilde\xi(k)$ and $\widetilde{\xi}^\ast(k)$. Hence, we
obtain another definition of $J_k$ equivalent to (\ref{15.h}):
\[
S(k) = S^\uparrow(J_k).
\]
Let $T(j)$ be the $T$ value corresponding to the order statistic
$S^\uparrow(j)$ of the sequence $S$. That is,
\[
T(j) = T^\uparrow(K_j),
\]
where $(K_j)$ is a random permutation of the positive integers. By
(\ref
{16.a}), $(J_k)$ is the inverse of $(K_j)$. Together with (\ref
{eqn:fidi.order}) and (\ref{15.a}), this implies (\ref{15.g}), proving
the last statement. The intensity measure $\mu$ comes from direct computation.
\end{pf*}

Marginal distributions of the random permutation $(J_k)$ and its
inverse $(K_j)$ are given in (\ref{eqn:kj}) and (\ref{eqn:jk}). Note
that for $k = 1,2, \ldots,$
\[
S(k) = T^\uparrow(k)\gamma_k^a/\Gamma(a+1)
\]
for i.i.d. $\gamma_k$ distributed as $\operatorname{gamma}(a+1,1)$, independent of the
sequence $T^\uparrow$, and
\[
T(k) = \Gamma(a+1)S^\uparrow(k)\tilde{\varepsilon}_k^{-a}
\]
for i.i.d. standard exponentials $\varepsilon_k$, independent of the
sequence $(S{(k)})$ but not of the $\gamma_k$.
Since the projection of a Poisson process is Poisson, the $s$ and
$t$-marginal of $\mu$ is just Lebesgue measure, as seen in the proof.
Thus by conditioning on either $S^\uparrow(k)$ or $T^\uparrow(k)$, one
can evaluate (\ref{eqn:kj}) and (\ref{eqn:jk}) explicitly. In
particular, by a change of variable $r = \Gamma(a+1)^{1/a}(s/t)^{1/a}$,
one can write $\mu$ in product form. This leads to the following.

%
\begin{prop}\label{prop:compute.general}
For $j \geq1$, conditioned on $S^\uparrow(j) = s, T(j) = \Gamma(a+1)
s r^{-a}$ for some $r > 0$, $K_j-1$ is distributed as
%
\begin{equation}
\label{eqn:marginal.kj} \operatorname{Poisson}\bigl(m(s,r)\bigr) + \operatorname{Binomial}\bigl(j-1, p(s,r)\bigr)
\end{equation}
with
%
\begin{equation}
\label{eqn:ppp.density.kj} m(s,r) = asr^{-a} \int_r^\infty
x^{a-1}\mathrm{e}^{-x} \,\mathrm{d}x
\end{equation}
and
%
\begin{equation}
\label{eqn:psr} p(s,r) = as^{2/a-2}r^{a-2}\int
_0^r x^{a-1}\mathrm{e}^{-x} \,\mathrm{d}x,
\end{equation}
where the $\operatorname{Poisson}$ and $\operatorname{Binomial}$ random variables are independent.
%
Similarly, for $k \geq1$, conditioned on $T^\uparrow(k) = t, S(k) = t
r^a / \Gamma(a+1)$ for some $r > 0$, $J_k - 1$ is distributed as
%
\begin{equation}
\label{eqn:marginal.jk} \operatorname{Poisson}\bigl(m'(t, r)\bigr) + \operatorname{Binomial}\bigl(k-1,
p'(t,r)\bigr)
\end{equation}
with
%
\begin{equation}
\label{eqn:ppp.density.jk} m'(t,r) = t \biggl(\frac{r^a + a \int_r^\infty x^{a-1}\mathrm{e}^{-x}
\,\mathrm{d}x}{\Gamma(a+1)} - 1 \biggr)
\end{equation}
and
%
\begin{equation}
\label{eqn:ptr} p'(t,r) = \Gamma(a+1)^{1-2/a}
at^{2/a-2}r^{a-1/a}\int_0^r
x^{a-1}\mathrm{e}^{-x} \,\mathrm{d}x,
\end{equation}
where the $\operatorname{Poisson}$ and $\operatorname{Binomial}$ random variables are independent.
\end{prop}

%

%
\begin{prop}[(Marginal distributions of $\bolds{K_1}$ and $\bolds{J_1}$)]\label{prop:compute}
Suppose that (\ref{eqn:assump.F}) holds for some $\lambda> 0$ and $a =
1$. Then the distribution of $K_1$, the $k$ such that $\xi(1) = \xi
^\ast
(k)$, is a mixture of geometric distributions, and so is that for
$J_1$, the $j$ such that $\xi^\ast(1) = \xi(j)$. In particular,
%
\begin{equation}
\label{18.k1} P(K_1 = k) = \int_0^\infty
p_rq_r^{k-1}\mathrm{e}^{-r} \,\mathrm{d}r,
\end{equation}
where $p_r = r/(r + \mathrm{e}^{-r})$, $q_r = 1 - p_r$, and
%
\begin{equation}
\label{18.j1} P(J_1 = j) = \int_0^\infty
\tilde{p}_r\tilde{q}_r^{j-1}r\mathrm{e}^{-r} \,\mathrm{d}r,
\end{equation}
where $\tilde{p}_r = 1/(r+\mathrm{e}^{-r})$, $\tilde{q}_r = 1 - \tilde{p}_r$.
\end{prop}

\begin{pf} When $a = 1$, $\int_r^\infty t^{a-1}\mathrm{e}^{-t} \,\mathrm{d}t =
\mathrm{e}^{-r}$. Substitute to (\ref{eqn:ppp.density.kj}) and (\ref
{eqn:ppp.density.jk}) give
\[
m(s,r) = sr^{-1}\mathrm{e}^{-r}, \qquad m'(t,r) = t
\bigl(r-1+\mathrm{e}^{-r}\bigr).
\]
By a change of variable, (\ref{eqn:measure.mu}) becomes
\[
\frac{\mu(\mathrm{d}s \,\mathrm{d}r)}{\mathrm{d}s \,\mathrm{d}r} = s\mathrm{e}^{-r}, \qquad\frac{\mu(\mathrm{d}t
\,\mathrm{d}r)}{\mathrm{d}t \,\mathrm{d}r} = tr\mathrm{e}^{-r}.
\]
Thus, \emph{conditioned on $s$ and $r$}, $K_1 - 1$ is distributed as
the number of points in a p.p.p. with rate $r^{-1}\mathrm{e}^{-r}$ before the
first point in a p.p.p. with rate 1. This is the geometric
distributions on $(0, 1, \ldots)$ with parameter $p_r = 1/(1 +
r^{-1}\mathrm{e}^{-r})$. Since the marginal density of $r$ is $\mathrm{e}^{-r}$,
integrating out $r$ gives (\ref{18.k1}). The computation for the
distribution of $J_1$ is similar.
\end{pf}

%
One can check that each of (\ref{18.k1}) and (\ref{18.j1}) sum to 1. We
conclude with a `fun' computation. Suppose that (\ref{eqn:assump.F})
holds for some $\lambda> 0$ and $a = 1$. That is, $F$ behaves like an
exponential c.d.f. near $0$. By Proposition~\ref{prop:compute}, $E(J_1)
= 9/4$ and $E(K_1) = \infty$. That is, the last size-biased pick is
expected to be almost the second smallest order statistic, while the
smallest order statistic is expected to be picked infinitely earlier on
in a successive sampling scheme(!). The probability that the last
species to be picked in a successive sampling scheme is also the one of
smallest species size is
\begin{eqnarray*}
\lim_{n \to\infty} P\bigl(X_n^{\mathrm{rev}}[1] =
X_n(1)\bigr) &=& P\bigl(\xi^\ast(1) = \xi(1)\bigr) =
P(J_1 = 1) = P(K_1 = 1)
\\
&= &\int_0^\infty\frac{r\mathrm{e}^{-r}}{r + \mathrm{e}^{-r}} \,\mathrm{d}r = 1 - \int
_0^1 \frac
{u}{u - \log u} \,\mathrm{d}u \approx0.555229.
\end{eqnarray*}

\section{Summary}
This paper reviewed and complemented results on the exact and
asymptotic distribution of the size-biased permutation of finitely many
independent and identically distributed positive terms. Our setting
lies in the intersection between induced order statistics, size-biased
permutation of ranked jumps of a subordinator, and successive sampling.
We discussed size-biased permutation from these different viewpoints
and obtained simpler proofs of known results. Our main contribution,
Theorem~\ref{thm}, gives a Poisson coupling between the asymptotic
distribution of the last few terms of a size-biased permutation and its
few smallest order statistics. 

We thank two anonymous referees for their careful reading and
constructive comments.\looseness=1


%

\printhistory
\end{document}